\documentclass[smallextended,numbook,runningheads]{svjour3}     
\smartqed  
\usepackage{graphicx}
\usepackage{amsmath}
\usepackage{epstopdf} 
\usepackage{mathptmx}      
\usepackage{amssymb}
\usepackage{bm}
\usepackage{hyperref}
\usepackage{xcolor}

\newcommand{\newtext}[1]{\textcolor{black}{#1}}

\newcommand\esssupp{\text{ess supp}}

\journalname{BIT}
\begin{document}

\title{Well-posedness of the Stokes problem under modified pressure Dirichlet boundary conditions}

\titlerunning{Stokes well-posedness under pressure Dirichlet conditions}        

\author{Igor Tominec \and
        Josefin Ahlkrona \and
        Malte Braack 
}


\institute{I. Tominec \at
Stockholm University, Department of Mathematics, SE-106 91 Stockholm, Sweden \\
              \email{igor.tominec@math.su.se}           
           \and
           J. Ahlkrona \at
           Stockholm University,  Department of Mathematics, SE-106 91 Stockholm, Sweden \\
           \email{ahlkrona@math.su.se}
           \and
           M. Braack \at
           Kiel University, Mathematical Seminar, DE-24098 Kiel, Germany\\
           \email{braack@math.uni-kiel.de}
           }

\date{Received: date / Accepted: date}

\maketitle
\begin{abstract}
  This paper shows that the Stokes problem is well-posed when velocity and pressure simultaneously vanish on the domain boundary. This result is achieved by extending Nečas' inequality to square-integrable functions that vanish in a small band covering the boundary. It is found that the associated a priori pressure estimate depends inversely on the volume of the band. Numerical experiments confirm these findings. Based on these results, guidelines are provided for applying vanishing pressure boundary conditions in model coupling and domain decomposition methods.  
\end{abstract}

\section{Introduction}

\noindent
The Stokes problem is a model of viscous fluid flow given by:
\begin{equation} 
    \label{eq:pstokes_strong}
    \begin{aligned}
 - \Delta \bm u + \nabla p &= \bm f & \text{ on } \Omega, \\
        \nabla \cdot \bm u &= \bm 0 & \text{ on } \Omega,
    \end{aligned}
\end{equation}
where $\Omega \subset \mathbb{R}^d$ is an open and bounded Lipschitz-continuous domain, 
$\bm u: \Omega \to \mathbb{R}^d$ is the velocity field,  
$p: \Omega \to \mathbb{R}$ is the pressure field, and $\bm f: \Omega \to \mathbb{R}^d$ 
is the external force field. To obtain a unique solution continuously dependent on the problem data (well-posedness),
additional constraints on velocity and pressure are required. 
A classical choice of constraints is Dirichlet boundary conditions on velocity over $\partial\Omega$ 
and a zero average condition on pressure over $\Omega$. Under such conditions, the problem is well-posed \cite{Girault_and_Raviart,volker_john,brenner_scott_book}. 


There exist cases where it is desirable to instead set Dirichlet-type conditions for both velocity and pressure on the boundary. 
These conditions are typically not motivated by physics but are useful for instance in 
domain decomposition methods \cite{domain_decomposition_toselli2006}. 
In domain decomposition methods the computational domain is split into smaller subdomains. Stokes-type problem is then solved on each subdomain, where 
the subproblems are coupled through boundary conditions to recover the solution over the whole domain. 
\newtext{For instance, \cite[Section 4.1]{PAVARINO200035} studies an overlapping Schwarz-type domain decomposition method, where Dirichlet values are prescribed in velocity and pressure on each subdomain.
Another example is adaptive PDE model refinement methods \cite{malte_model_refinement,oden_model_refinement,ISSM_iscal}, which integrate solutions from models of varying physical fidelity to optimize computational efficiency. For instance, in \cite{iscal} the Shallow Ice Approximation (SIA) model is employed across the entire domain, followed by refining the solution in regions  where the physics lacks accuracy, by solving the more detailed Stokes problem over a subdomain. The Dirichlet velocity and pressure boundary conditions for this subdomain are obtained from the SIA solution, which couples the solutions over the subdomain and the rest of the domain. }
In this paper, we analyze these boundary conditions from a mathematical perspective and propose potential improvements.


\newtext{A potential use case are multigrid methods on hierarchically refined meshes for finite element methods \cite{becker_braack_multigrid}. In these methods, multigrid smoothing (local smoothing) is performed on each refinement level to compute locally defined solutions. The local smoothing problem uses Dirichlet boundary conditions that couple to the surrounding local solutions and by that ensure coupling across all refinement levels. }

To our knowledge, the well-posedness of the Stokes problem under velocity and pressure Dirichlet conditions has not yet been 
studied. Gresho and Sani \cite{GreshoSani1987} note that while pressure boundary conditions are commonly used, they should be avoided when possible.
In \cite{Conca,BERTOLUZZA201758,bernardi:hal-00961653,seloula:hal-00686230} 
$p|_{\partial\Omega} = 0$ is combined with conditions on $\bm u \times \bm n$, where $\bm n$ is an outward normal to $\partial\Omega$. 
Without these modifications, $p|_{\partial\Omega} = 0$ does not lead to well-posedness as the pressure is generally an 
$L^2(\Omega)$ (discontinuous) function, and 
the trace operator from $L^2(\Omega)$ to $L^2(\partial \Omega)$ is not well defined. 
For this reason, we propose relaxed Dirichlet-type boundary conditions:
\begin{equation}
    \label{eq:bc_u_p_homogenous_reformulated}
\bm u\big|_{\partial\Omega} = 0 \text{  and  } p|_\Delta = 0,
\end{equation}
and prove the Stokes problem well-posedness. 
Here, $\Delta \subset \Omega$ is a band with a small thickness, 
placed around $\partial\Omega$ as shown in Figure \ref{fig:intro:domains_delta}. 
As $|\Delta| \to 0$ the pressure part of the modified boundary conditions is approaching the $p|_{\partial\Omega} = 0$ case. 
\begin{figure}[t]
    \centering
\begin{tabular}{ccc}
    \multicolumn{3}{c}{\textbf{Domain $\bm{\Omega}$ and the embedded $\bm{\Delta}$ regions}} \\
    \hspace{-0.3cm}\textbf{Region $\bm \Delta = \bm \emptyset $} &
    \textbf{Region $\bm \Delta$ along boundary} &
    \textbf{Region $\bm \Delta$ in interior} 
    \\
    \includegraphics[width=0.25\linewidth]{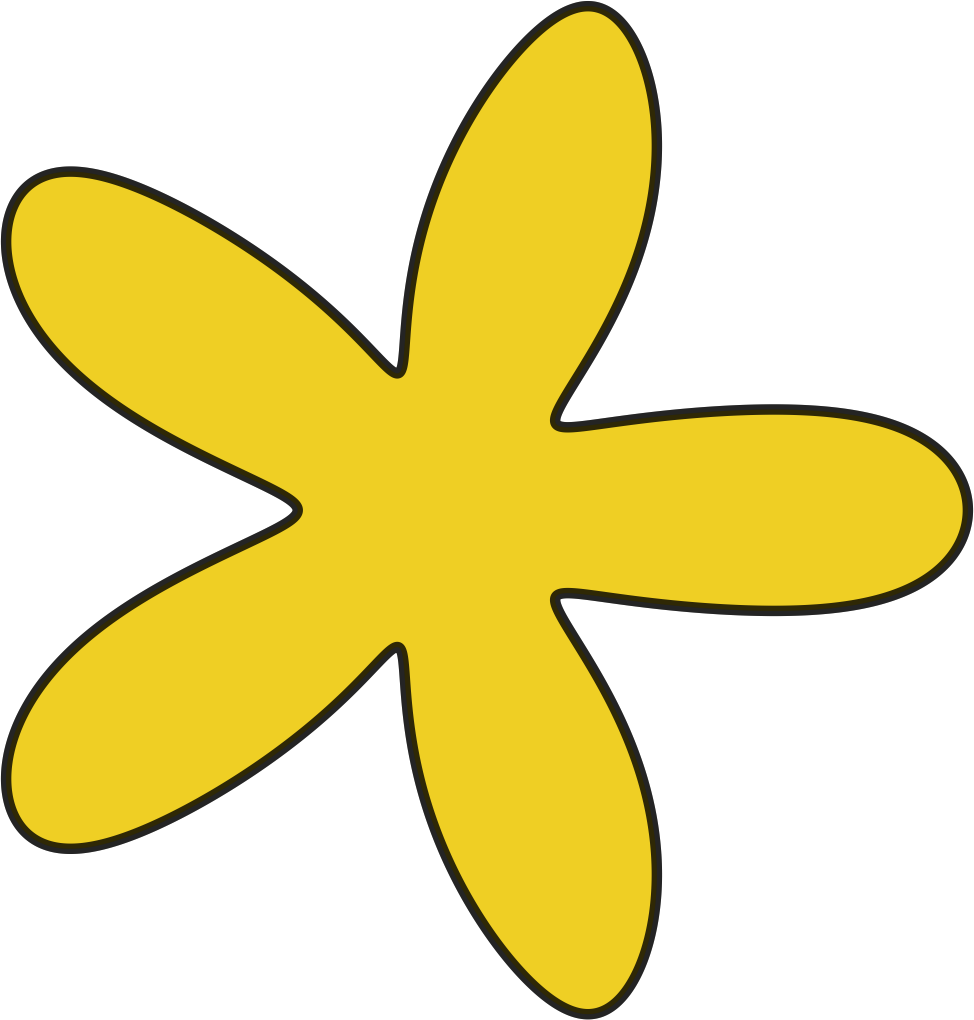} &
    \includegraphics[width=0.25\linewidth]{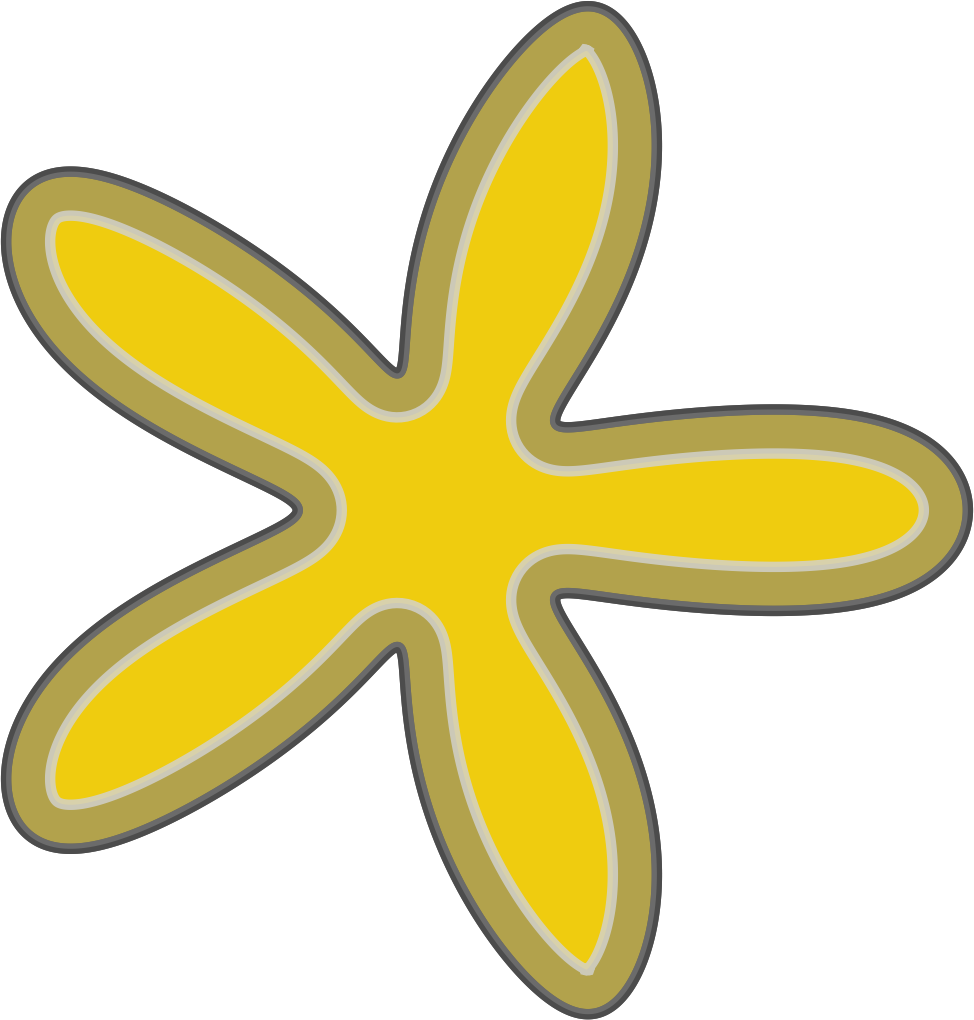} &
    \includegraphics[width=0.25\linewidth]{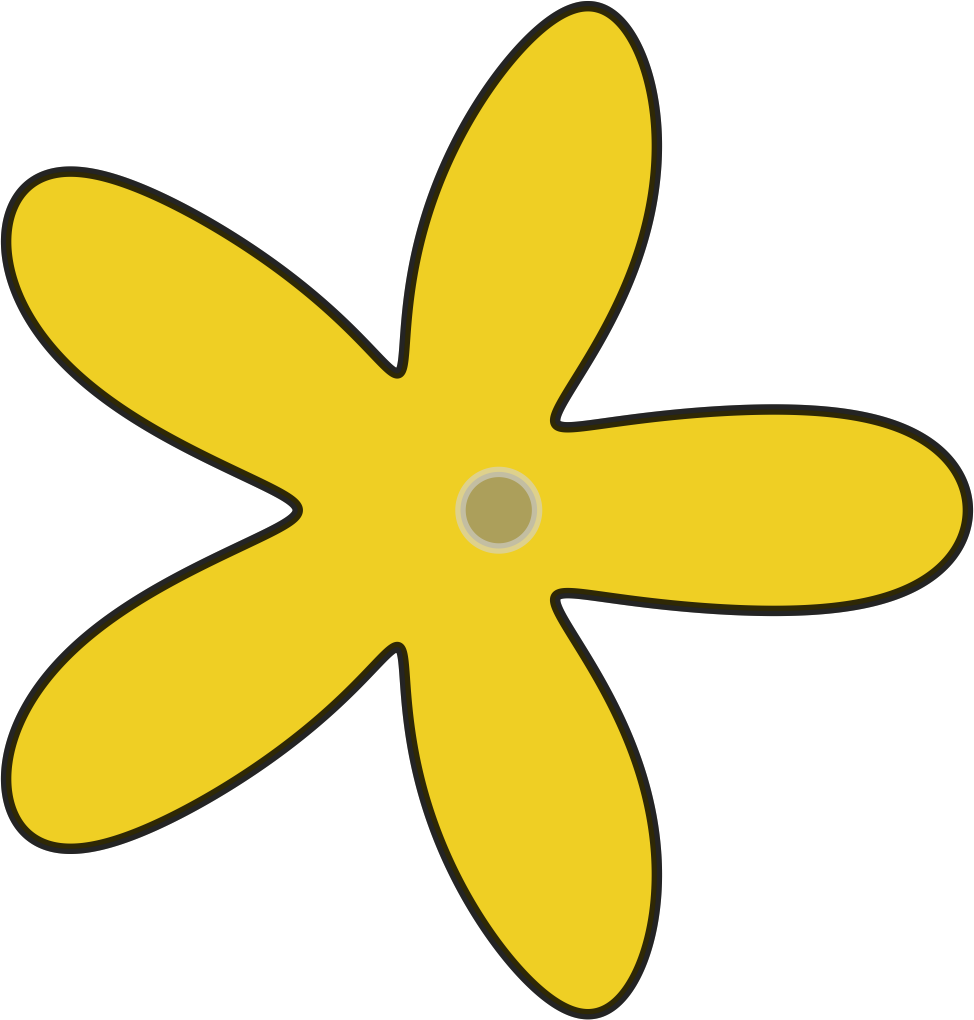}
\end{tabular}
\caption{The flower shaped domain is $\Omega$. Different choices of region $\Delta \subset \Omega$ are shaded using grey color. In the three images from left to right: $\Delta$ is a zero set, $\Delta$ is a small band around the boundary of $\Omega$, and $\Delta$ is a small patch in the interior of $\Omega$.}
\label{fig:intro:domains_delta}
\end{figure}

Our goal is to establish mathematical limitations associated with using boundary conditions \eqref{eq:bc_u_p_homogenous_reformulated}. To this end, we provide guidance on selecting the size of 
 $\Delta$ by deriving an a priori pressure estimate. Specifically, we extend the Nečas inequality \cite{Girault_and_Raviart} from $L^2(\Omega)$
pressure functions with zero average to those that vanish over a subset $\Delta \subset \Omega$.
While this extension is a secondary result, it is novel in itself.
We leverage the extended Nečas inequality to establish the inf-sup property of the Stokes problem’s bilinear form involving the pressure term. Our theoretical findings hold in infinite-dimensional solution spaces. To complement these results, we present a set of numerical experiments where $\bm u$ and $p$
are discretized using quadratic and linear Lagrange finite elements, respectively (Taylor-Hood elements).

Our results are also applicable to understanding the effects of a common practice: setting $p=0$ 
in one point on $\partial\Omega$ or one point inside $\Omega$ as a replacement for imposing a zero average pressure condition. 
The latter fits the generalized boundary condition formulation 
\eqref{eq:bc_u_p_homogenous_reformulated} with a small $\Delta$-ball placed in the domain interior as shown in the right-most plot of Figure \ref{fig:intro:domains_delta}.

\newtext{In Section \ref{sec:preliminaries}, we outline the mathematical preliminaries. We state the problem formulation in Section \ref{sec:problem_formulation}. 
In Section \ref{sec:discussion_require_p_on_bnd} we motivate imposing $p|_\Delta = 0$ instead of $p|_{\partial\Omega} = 0$.
In Section \ref{sec:theory_overview} we provide an overview of steps leading to the problem well-posedness. 
In Section \ref{section:operators_P_and_T} and in Section \ref{section:Stokes_wellposedness_pressure_L_2_00} we extend the Nečas inequality to $L^2_\Delta(\Omega)$, and then show the Stokes problem well-posedness when $p \in L^2_\Delta(\Omega)$. Furthermore, we prove the $L^1$-norm velocity divergence estimate.
We numerically verify all the estimates in Section \ref{section:numerical_experiments}. 
In Section \ref{sec:guidelines} we use our results to provide guidelines on imposing Dirichlet-type pressure boundary conditions.
In Section \ref{sec:finalremarks} we give our final remarks.}

\section{Problem formulation}\label{sec:problem_formulation0}
\subsection{Preliminaries: function spaces and norms}\label{sec:preliminaries}

Throughout this paper, we conduct the theoretical analysis of the Stokes problem \eqref{eq:pstokes} in infinite-dimensional function spaces defined on an open, bounded, and Lip\-schitz-continuous domain 
 $\Omega \subset \mathbb{R}^d$. The space $L^2(\Omega)$  is the standard Lebesgue space of square-integrable functions over $\Omega$
with inner product $(\cdot,\cdot)$ and norm $\|\cdot\|_{L^2(\Omega)}^2$, see e.g. \cite{guermond_voli}.
 A key subspace of $L^2(\Omega)$  is the space of functions with zero average, defined as
\begin{equation}
\label{eq:zeroaverage_pressure_space}
   L^2_0(\Omega) = \Big\{ q \in L^2(\Omega) \Big| \int_\Omega q\, d\Omega = 0 \Big\}.
\end{equation}
The central functional space in this work is $L^2_\Delta(\Omega)$, which consists of functions in $L^2(\Omega)$ 
 that vanish on an open subset $\Delta \subset \Omega$:
 \begin{equation}
    \label{eq:essentially_supported_pressure_space}
    L^2_\Delta(\Omega) = \left\{q \in L^2(\Omega)\, \Big|\, \esssupp(q) = \overline{\Omega\backslash \Delta}
    \right\}.
\end{equation}
Here, $\esssupp(q)$ is essential support of $q$ defined as:
\begin{equation}
    \label{eq:essential_support}
    \esssupp(q) = \overline{\Omega \backslash \bigcup \left\{ B_\varepsilon \subseteq \Omega: B_\varepsilon \text{ is open and } q=0 \text{ almost everywhere in } B_\varepsilon \right\}},
\end{equation}
where $B_\varepsilon$ is an open ball with radius $\varepsilon>0$. 
The space $\bm H^1(\Omega)$ and its restriction to vanishing velocities over $\partial\Omega$ are respectively 
defined through:
\begin{equation*}
    \begin{aligned}
\bm H^1(\Omega) &= \Big\{ \bm v \in [L^2(\Omega)]^d \, \Big|\, \|\bm v\|_{L^2(\Omega)}^2 + \|\nabla \bm v\|_{L^2(\Omega)}^2 < \infty \Big\}, \\
\bm H^1_0(\Omega) &= \Big\{ \bm v \in \bm H^1(\Omega) \, \Big|\, \bm v|_{\partial\Omega} = 0 \Big\}.
    \end{aligned}
\end{equation*}
$H^{-1}(\Omega)$ is a dual space to $H^1_0(\Omega)$  with inclusions $H^1_0(\Omega) \subset L^2(\Omega) \subset H^{-1}(\Omega)$ \cite{Evans_book}. An 
associated norm is defined as: 
$$\|q\|_{H^{-1}(\Omega)} = \sup_{\bm v \in \bm H^1_0(\Omega)}\, \frac{(q, \bm v)_{H^{-1}(\Omega), H^1_0(\Omega)}}{\|\bm v\|_{H^1_0(\Omega)}},$$
where $(q, \bm v)_{H^{-1}(\Omega), H^1_0(\Omega)}$ is the dual pairing between $H^{-1}(\Omega)$ and $H^1_0(\Omega)$.

In Section \ref{section:numerical_experiments},  we compute solutions to the Stokes problem using the finite element method with Taylor-Hood elements. 
The corresponding finite-dimensional velocity and pressure spaces, $\bm{V_h} \subset \bm H^1_0(\Omega)$ 
and $Q_h \subset L^2(\Omega)$ are defined as follows:
\begin{equation*}
    \begin{aligned}
\bm{V_h} &= \Big\{ \bm v_h \in \bm H^1_0(\Omega) \, \Big|\, \bm v_h \big |_{K} \in \mathbb{P}_2(K),\, \forall K \in \mathcal T_h \Big\}, \\
Q_h &= \Big\{ q_h \in L^2(\Omega) \, \Big|\, q_h \big |_{K} \in \mathbb{P}_1(K),\, \forall K \in \mathcal T_h \Big\}.
    \end{aligned}
\end{equation*}
Here, $\mathbb{P}^2$ and $\mathbb{P}^1$ denote multivariate ($d$ dimensional) polynomial spaces of order $2$ and $1$ respectively,
defined over each element $K$ of a mesh $\mathcal T_h \subset \Omega$. 
 In the numerical experiments, the pressure space The $Q_h$ is further intersected with either $L^2_\Delta(\Omega)$ or $L^2_0(\Omega)$
 to enforce  the desired pressure conditions.

\subsection{Weak form of Stokes problem}\label{sec:problem_formulation}

The weak formulation of the Stokes problem \eqref{eq:pstokes_strong} 
seeks $\bm u \in \bm V:= \bm H^1_0(\Omega)$ (velocity space) and $p \in Q:=L^2_\Delta(\Omega)$ 
(pressure space) such that:
\begin{equation} 
    \label{eq:pstokes}
    \begin{aligned}
        (\nabla \bm u, \nabla \bm v) - (p, \nabla \cdot \bm v)  &= (\bm f, \bm v), & \forall \bm v &\in \bm V, \\
        (\nabla \cdot \bm u, q) &= 0, & \forall q &\in Q,
    \end{aligned}
\end{equation}
where $(\bm v, q)$ are the test functions, and $\bm f \in \bm V'$ (velocity dual space) is a forcing function. 
In abstract form, the problem is to find $\bm u \in \bm V$ and $p \in Q$ such that:
\begin{equation}
    \label{eq:mixed_problem}
\begin{aligned}
a(\bm u, \bm v) + b(\bm v, p) &= (\bm f,\bm v),\,& \forall \bm v &\in \bm V,\\
b(\bm u, q) &= 0,\,& \forall q &\in Q,
\end{aligned}
\end{equation}
where $a(\bm u, \bm v)$ and $b(\bm v, p)$ are bilinear forms on $\bm V \times \bm V$ and $\bm V \times Q$, respectively.
The well-posedness of the above problem is given by the following theorem.
\begin{theorem}[Babuška-Brezzi, Theorem 49.3 in \cite{guermond_volii}]
    \label{theorem:babuska_brezzi}
Problem \eqref{eq:mixed_problem} is well-posed if:
\begin{subequations}
\begin{align}
    a(\bm u, \bm u) &\geq C_A \|\bm u\|_{\bm V},\,& \bm u \in \bm V, \label{eq:problem_formulation:a_coercive}\\
    a(\bm u, \bm v) &\leq C_B \|\bm u\|_{\bm V}\, \|\bm v\|_{\bm V}\, & \bm u \in \bm V,\, \bm v \in \bm V, \label{eq:problem_formulation:a_continuous}\\
    b (\bm u, q) &\leq  C_C \|\bm u\|_{V}\, \|q\|_{Q}\, & \bm u \in \bm V,\, q \in Q, \label{eq:problem_formulation:b_continuous}\\
    \sup_{\bm v \in \bm V} \frac{b(\bm u, q)}{\|\bm u\|_{\bm V}} &\geq C_D\, \|q\|_{Q}\, & q \in Q. \label{eq:problem_formulation:b_infsup}
\end{align}
\end{subequations}
\end{theorem}
This framework provides the basis for analyzing the well-posedness of the Stokes problem when $\bm V=\bm H^1_0(\Omega)$  and $Q=L^2_\Delta(\Omega)$,
 as defined in \eqref{eq:essentially_supported_pressure_space}. The
boundary conditions \eqref{eq:bc_u_p_homogenous_reformulated} are  incorporated  within these spaces. 
The key properties of $\Delta$ are: 
(i) $|\Delta| > 0$, 
(ii) $\Delta$ is an open set,
 (iii) $\Omega \backslash \Delta$ is a closed set when $\Delta$ is a band around $\partial\Omega$.

\subsection{Vanishing pressure on a strip covering $\partial\Omega$.}\label{sec:discussion_require_p_on_bnd}

We now motivate the choice of the vanishing pressure space
$L^2_\Delta(\Omega)$, which enforces the condition $p|_\Delta = 0$, over an alternative $L^2(\Omega)$-based space that imposes
$p|_{\partial\Omega} = 0$. 
 The latter condition leads to a loss of well-posedness for problem \eqref{eq:pstokes}, particularly in terms of uniqueness.
This issue arises due to:\\
(i) the null space of the gradient operator present in \eqref{eq:pstokes}, given by ($-(p, \nabla \cdot \bm v) = (\nabla p, \bm v)$) 
contains constant functions $p^* \in \mathbb{R}(\Omega)$, meaning, $\nabla p^* = 0$.\\
(ii) Imposing $p^*|_{\partial\Omega} = 0$ does not necessarily eliminate constant solutions, as shown in the left plot of Figure \ref{fig:intro:piecewise_discontinuous_constants}. 
This is because the trace operator is not continuous in $L^2(\Omega)$, meaning that a function vanishing on $\partial\Omega$
can still be constant in $\Omega$,
leaving the null space unchanged and breaking uniqueness.

On the other hand, enforcing $p^*$ to satisfy the boundary conditions \eqref{eq:bc_u_p_homogenous_reformulated}, 
ensures that $p^*$ is not a constant function on $\Omega$, as illustrated in the right  plot of Figure \ref{fig:intro:piecewise_discontinuous_constants}. 
In this case, the derivative at the jump discontinuity inside $\Omega$ is a generalized Dirac delta function, which is measurable in space $H^{-1}(\Omega)$.
 As a result, the null space of the gradient operator no longer contains constant pressure functions. 
 However, while this is a necessary step toward well-posedness, it is not a sufficient condition on its own.


\begin{figure}[t]
    \centering
\begin{tabular}{ccc}
    $p|_{\partial\Omega} = 0$ &
    $p|_{\Delta} = 0$
    \\
    \hspace{-0.1cm}\includegraphics[width=0.4\linewidth]{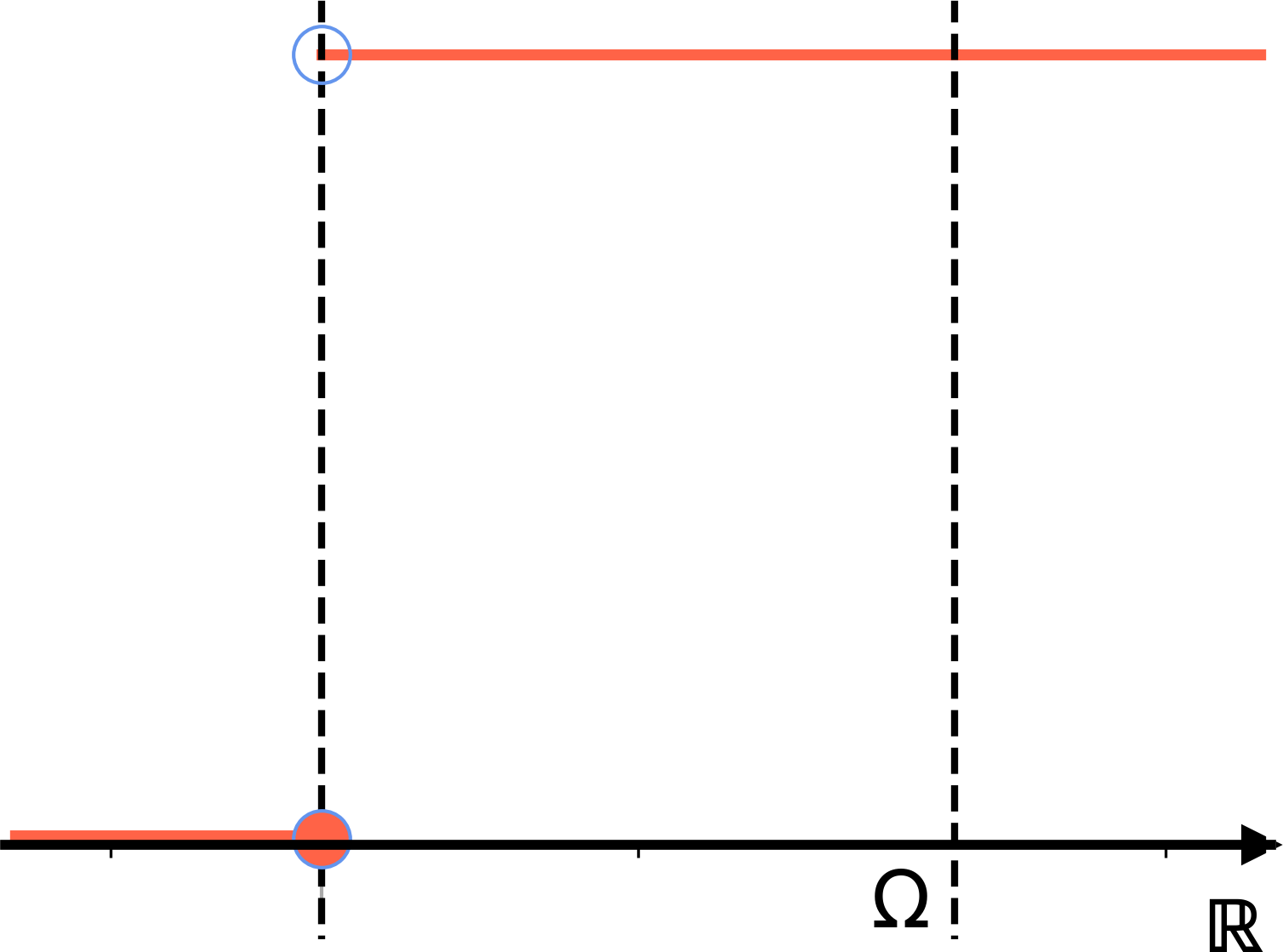} &
    \includegraphics[width=0.4\linewidth]{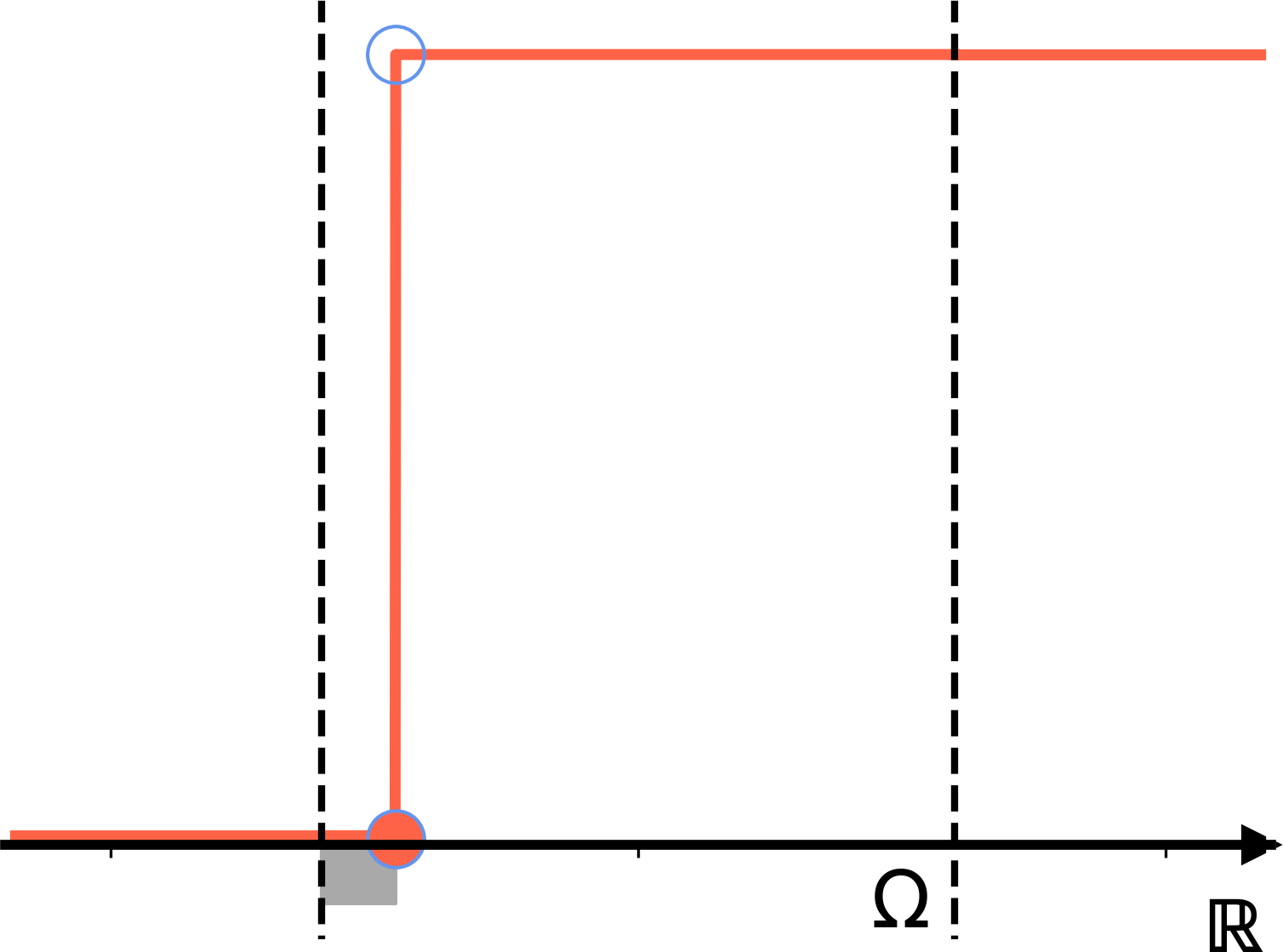}
\end{tabular}
\caption{Piecewise discontinuous constant pressure functions vanishing on one part of $\partial\Omega$ (left) and vanishing on the shaded area $\Delta \subset \Omega$ close to one part of $\partial\Omega$ (right), 
where $\Omega \subset \mathbb{R}$ is an open and bounded domain.}
\label{fig:intro:piecewise_discontinuous_constants}
\end{figure}

\section{Numerical analysis}
\subsection{An overview of the theoretical investigation}\label{sec:theory_overview}

When $\bm V = \bm H^1_0(\Omega)$ and $Q = L^2_0(\Omega)$, 
the Stokes problem  \eqref{eq:pstokes} is well-posed \cite{Girault_and_Raviart,volker_john,brenner_scott_book}, satisfying all the conditions of
Theorem \ref{theorem:babuska_brezzi}. 
In this case, the inf-sup (Babuška-Brezzi) condition \eqref{eq:problem_formulation:b_infsup} holds due to 
the inequality established by Nečas \cite{Girault_and_Raviart}:
\begin{equation}
    \label{eq:girault_raviart_inequality}
    \|p\|_{L^2(\Omega)} \leq C_1\, \|\nabla p\|_{H^{-1}(\Omega)},\qquad \forall p \in L_0^2(\Omega),
\end{equation} 
where $C_1$ depends only on $\Omega$. However, when  choosing 
$\bm V = \bm H^1_0(\Omega)$ and $Q =  L^2_\Delta(\Omega)$,  the well-posedness of \eqref{eq:pstokes} remains uncertain. 
Conditions \eqref{eq:problem_formulation:a_coercive}–\eqref{eq:problem_formulation:b_continuous} are satisfied, 
but the inf-sup condition \eqref{eq:problem_formulation:b_infsup} requires reassessment. 
To address this, we extend inequality \eqref{eq:girault_raviart_inequality} to all
$p \in L_\Delta^2(\Omega)$ and use  this extension to establish 
 \eqref{eq:problem_formulation:b_infsup} for $Q=L^2_\Delta(\Omega)$. 

We first demonstrate that 
$L^2_\Delta(\Omega)$ is complete and establish the following lower bound estimate:
\begin{equation}
    \label{eq:overview:lowerbound}
\|Tp\|_{L^2(\Omega)} \geq C_\Delta\, \|p\|_{L^2(\Omega)},\quad Tp \in L^2_0(\Omega),\, p \in L^2_\Delta(\Omega),
\end{equation}
where $T$ is  the average-subtracting operator, defined as $Tp = p - \frac{1}{|\Omega|} \int_\Omega p\, d\Omega$. 
Next, we combine \eqref{eq:overview:lowerbound}  with the Nečas inequality  \eqref{eq:girault_raviart_inequality} to  derive its extended form:
$$\|p\|_{L^2(\Omega)}\leq C_\Delta^{-1}\, C_1\, \|\nabla p \|_{H^{-1}(\Omega)},\quad \forall p \in L_\Delta^2(\Omega),$$
where $C_\Delta^{-1}$ depends on $\Delta$, and  in limiting cases where a sequence in $L^2_\Delta(\Omega)$ approaches a constant function,
$C_\Delta^{-1}\le |\Delta|^{-1/2}|\Omega|^{1/2}$.
Finally, we  proceed as follows: (i) use the extended Nečas inequality to establish the missing inf-sup condition \eqref{eq:problem_formulation:b_infsup} for $Q=L^2_\Delta(\Omega)$, 
(ii) apply Theorem \ref{theorem:babuska_brezzi} to argue for well-posedness, and  
(iii) derive the corresponding  well-posedness estimates.

The requirements for several theorems used in our theoretical argumentation are that the involved spaces be of Banach type. 
The spaces in question are $L^2_{0}(\Omega)$ from \eqref{eq:zeroaverage_pressure_space} and  $L^2_\Delta(\Omega)$
 from \eqref{eq:essentially_supported_pressure_space}. Both are subspaces of  $L^2(\Omega)$, so the  $L^2$ -norm is applicable to both spaces. 
\newtext{It is well-known that $L_0^2(\Omega)$, as a closed subspace of $L^2(\Omega)$, is complete. As of completeness of $L^2_\Delta(\Omega)$: for any $f\in L^2_\Delta(\Omega)$ we have the restriction $\tilde f:=f|_{\Omega\setminus\Delta}\in L^2(\Omega\setminus\Delta)$, and, vice-versa,
any $\tilde f\in L^2(\Omega\setminus\Delta)$ with trivial extension to $f\in L^2(\Omega)$ is
also in $L^2_\Delta(\Omega)$. Moreover, $\|\tilde f\|_{L^2(\Omega\setminus\Delta)}=\|f\|_{L^2(\Omega)}$, so that
the space $L^2_\Delta(\Omega)$ is isomorph to the complete
space $L^2(\Omega\setminus\Delta)$. Thus, $L^2_\Delta(\Omega)$ is complete.
}
\subsection{\newtext{Operators between $L^2$-subspaces}}\label{section:operators_P_and_T}

We use the operator $P$ to project an arbitrary function from $L^2(\Omega)$ to 
$L^2_\Delta(\Omega)$. Additionally, we use the operator $T$ to subtract the average of a  function from $L^2_\Delta(\Omega)$,
thereby projecting  it to $L^2_0(\Omega)$. Both operators are employed to extend the Nečas inequality \eqref{eq:girault_raviart_inequality} 
to the space $L^2_\Delta(\Omega)$. The definitions of these operators 
are stated below.

\begin{definition}[\newtext{Operator $P$ and $T$}]
    \label{definition:operator_P}{\em
    Let $I_\Delta$ be the indicator function w.r.t. $\Delta$.
    We define the projection $P : L^2(\Omega) \to L^2_\Delta(\Omega)$,   the average mapping $M: L^2_\Delta(\Omega) \to L^2(\Omega)$,
    and the average subtracting mapping $T: L^2_\Delta(\Omega) \to L^2_{0}(\Omega)$ by:
    \begin{eqnarray}
       Pq &:=& q-I_\Delta q,\label{eq:operator_P}\\
       Mq &:=& \frac{1}{|\Omega|} \int_{\Omega} q \, d\Omega,\label{eq:operator_M}\\
       Tq &:=& q -Mq.\label{eq:operator_T}
    \end{eqnarray}}
\end{definition}

We now outline the fundamental properties of these operators. The main goal in this section is to establish that 
$T$ is bounded from below, a result that will later be used to extend the Nečas inequality \eqref{eq:girault_raviart_inequality} to the case when 
$p \in L^2_\Delta(\Omega)$.
 
\begin{lemma}[\newtext{Properties of operator $T$}]
    \label{lemma:T_continuous}
    The linear operator $T: L^2_\Delta(\Omega) \to L_0^2(\Omega)$ given by (\ref{eq:operator_T}) is continuous and injective.
    For the adjoint holds $T^*=P$.
\end{lemma}
\begin{proof}
\emph{(i) $T$ is continuous.}
$M$ is continuous, since  
    $\| Mq \|_{L^2(\Omega)} \leq |\Omega|^{-1/2} \|Mq\|_{L^1(\Omega)}\leq \|q\|_{L^2(\Omega)}$ for  $q \in L^2_\Delta(\Omega)$.
    Thus $T$ is continuous.\\[2mm]
\emph{(ii) T is injective.} We notice that $Tq= 0$ iff
$q =M q$. But such $q$ is a constant. 
Taking $q \in L_\Delta^2$, $Tq=0$ is then possible only when $q=0$, as $L^2_\Delta(\Omega)$ by construction contains 
functions with essential support smaller than $\Omega$ (see \eqref{eq:essentially_supported_pressure_space}), but a constant 
function has an essential support equal $\Omega$. Thus, $T$ is injective.\\[2mm]
\noindent
\emph{(iii) Adjoint of $T$ is $P$.} For $f \in L^2_0(\Omega)'\simeq L_0^2(\Omega)$ it holds $Mf=0$. Hence, for $g\in L^2_\Delta(\Omega)$ we get
\begin{equation*}
\begin{aligned}
     \langle T^*f,g\rangle_{L^2_\Delta(\Omega)',L^2_\Delta(\Omega)} = (f,Tg) = (f,g)-(f,Mg) &= \int_\Delta fg\, d\Omega-|\Omega|\,Mf\, Mg \\
     &= \int_\Delta fg\, d\Omega =( Pf,g).
 \end{aligned}
\end{equation*}
This implies $T^*=P$.
\end{proof}

\begin{lemma}[\newtext{Properties of operator $P$}]
    \label{lemma:P_nullspace}
The null space of $P:L^2_0(\Omega) \to L^2_\Delta(\Omega)$, defined in (\ref{eq:operator_P}), and its orthogonal complement are given by:
$$N(P) = \{ f \in L_0^2(\Omega)\, |\,  f=0\text{ on } \Omega \backslash \Delta \},$$
    $$N(P)^\perp  =\{ g \in L_0^2(\Omega)\, |\, g|_{\Delta} = C\in\mathbb{R}\}.$$
\end{lemma}
\begin{proof}
\emph{(i) Nullspace of $P$.}
$Pf=0$ for  $f \in L^2_0(\Omega)$ iff $f = I_\Delta f$. This is equivalent to
$f=0$ on $\Omega \backslash \Delta$. \\[3mm]
\emph{(ii) Orthogonal complement-nullspace of $P$.}
The orthogonal complement takes the form: 
$N(P)^\perp = \{ g \in L_0^2(\Omega)\, |\, (g,f)=0\quad \forall f\in N(P) \}.$ 
Hence, $g\in N(P)^\perp$ is equivalent to
$$
   \int_\Delta fg\, d\Omega = 0\quad\forall f\in L^2_0(\Omega)\mbox{ with }\mbox{supp}\, f\subset \overline{\Delta}.
$$
This equality holds, iff $g\in L^2_0(\Delta)^\perp$. Because $L^2_0(\Delta)^\perp$
consists of all constant functions on $\Delta$, we obtain $N(P)^\perp  =\{ g \in L_0^2(\Omega)\, |\, g|_{\Delta} = C\in\mathbb{R}\}$.
\end{proof}

\begin{lemma}[\newtext{Range and kernel of $T$}] \label{lemma:T_range_closed}
The range of operator $T: L^2_\Delta(\Omega) \to L_0^2(\Omega)$ is given by
$$R(T) = \{g \in L^2_0(\Omega)\, |\, g = C \in \mathbb{R}\, \text{ on }\, \Delta\},$$
and is closed in $L_0^2(\Omega)$. The following lower bounded holds
\begin{eqnarray}\label{eq:Tinvcontinuous}
   \|Tf\|_{L^2(\Omega)} &\geq& C_\Delta\, \|f\|_{L^2(\Omega)}\quad\forall f\in L^2_\Delta(\Omega),
\end{eqnarray}
with a constant 
$C_\Delta \ge   (|\Delta|/|\Omega|)^{1/2}>0$.
\end{lemma}
\begin{proof}
\emph{(i) Range of T.} For any $f \in L^2_\Delta(\Omega)$,
$
   (Tf)|_\Delta = f|_\Delta - Mf = -Mf
$
is a constant.\\[2mm]
\emph{(ii) Range of T is closed.}
$T:L^2_\Delta(\Omega) \to L_0^2(\Omega)$ is linear and continuous. Since $T^*=P$, we have with Lemma \ref{lemma:T_continuous}
and \ref{lemma:P_nullspace}: $N(T^*)^\perp=N(P)^\perp=R(T)$. By the Closed Range theorem \cite{Yosida1978}, this implies that $R(T)$ is closed.\\[2mm]
\emph{(iii) T is bounded from below.}
As $T :L^2_\Delta(\Omega) \to R(T)$ is surjective by definition, and injective by Lemma \ref{lemma:T_continuous}), it is a continuous bijective
linear operator.  As a consequence of the Bounded Inverse Theorem \cite{Yosida1978},  the inverse $T^{-1}:R(T)\to L^2_\Delta(\Omega)$ is continuous, and \eqref{eq:Tinvcontinuous} follows, see \cite{ErnGuermond2004} (Lemma A.36).
Moreover, we get by elementary calculus the equality:
\begin{eqnarray*}
  \|Tf\|^2_{L^2(\Omega} &=& \|f - Mf\|^2_{L^2(\Omega} =  \|f\|^2_{L^2(\Omega)} - |\Omega||Mf|^2,
\end{eqnarray*}
and for the mean value the upper bound:
\begin{eqnarray*}
  |\Omega| |Mf|&\le& \|f\|_{L^1(\Omega\setminus\Delta)} \le 
  (|\Omega|-|\Delta|)^{1/2} \|f\|_{L^2(\Omega\setminus\Delta)} \\
  &\le& (1-|\Delta|/|\Omega|)^{1/2}\|f\|_{L^2(\Omega)}.
\end{eqnarray*}
This yields the lower bound:
\begin{eqnarray*}
   \|Tf\|^2_{L^2(\Omega} &\ge& (|\Delta|/|\Omega|)\|f\|_{L^2(\Omega)}^2.
\end{eqnarray*}
Hence, the constant $C_\Delta$ is bounded from below by:
$$
   C_\Delta \ge   (|\Delta|/|\Omega|)^{1/2}.
$$
\end{proof}

\subsection{\newtext{Well-posedness and stability of the Stokes problem}}\label{section:Stokes_wellposedness_pressure_L_2_00}

We now extend the Nečas inequality \eqref{eq:girault_raviart_inequality} from 
$L^2_0(\Omega)$ to $L^2_\Delta(\Omega)$.
\begin{theorem}[\newtext{Extended Nečas inequality}]
    \label{theorem:finalresult_necas_extended}
Let $\Omega$ be a bounded, Lipschitz domaint. Then
$$ \|f\|_{L^2(\Omega)} \leq C_\Delta^{-1} C_1\, \|\nabla f\|_{H^{-1}(\Omega)}\quad \forall f \in L^2_\Delta(\Omega),$$
Here $C_1>0$ only depends on $\Omega$, and $C_\Delta^{-1}$ depends on $\Delta$.
\end{theorem}
\begin{proof}
From Lemma \ref{lemma:T_range_closed} we have $\|f\|_{L^2(\Omega)} \leq C_\Delta^{-1}\, \| Tf \|_{L^2(\Omega)}$ with $ C_\Delta > 0$.
Since $Tf\in L^2_0(\Omega)$ we can apply
the original Nečas inequality \eqref{eq:girault_raviart_inequality} to obtain $\| Tf \|_{L^2(\Omega)}\le C_1\, \| \nabla (Tf) \|_{H^{-1}(\Omega)}$.
Since $Tf=f-Mf$ we obtain 
$$
\|f\|_{L^2(\Omega)} \leq C_\Delta^{-1}C_1\,(\| \nabla f\|_{H^{-1}(\Omega)} +\| \nabla(Mf) \|_{H^{-1}(\Omega)}).
$$
For $\bm v\in \bm V$ it holds $\langle \nabla Mf, \bm v\rangle = -(Mf, \nabla \cdot \bm v) = -Mf \int_{\partial\Omega} \bm v \cdot \bm n\, ds = 0$.
\end{proof}
Finally, we apply the extended Nečas inequality for $L^2_\Delta(\Omega)$
 from Theorem \ref{theorem:finalresult_necas_extended} to establish the well-posedness of the Stokes problem under the modified boundary conditions \eqref{eq:bc_u_p_homogenous_reformulated}. The result is stated in the theorem below.
 
\begin{theorem}[\newtext{Well-posedness}]
    \label{theorem:stokes_wellposedness_final}
Let $\Omega$ and $\Delta$ be bounded Lipschitz domains with $\Delta\subset\Omega$, $\partial\Omega\subset\overline\Delta$ and $|\Delta|>0$. 
The Stokes problem \eqref{eq:pstokes} is well-posed 
for $p \in L^2_\Delta(\Omega)$ and $\bm u \in \bm H_0^1(\Omega)$.
In addition, the following a priori estimates hold:
\begin{eqnarray*}
    \| p \|_{L^2(\Omega)} &\leq& 2\,C_\Delta^{-1}\, C_1\,\max(1,C_p^{-1}) \|\bm f \|_{H^{-1}(\Omega)} ,\\
   \| \bm u \|_{H_0^1(\Omega)}  &\leq& C_p^{-1}\, \|\bm f \|_{H^{-1}(\Omega)},
\end{eqnarray*}
where $C_\Delta > 0$ depends on $\Delta$, and $C_1,C_p>0$ are depending on $\Omega$.
\end{theorem}

\begin{proof}
We set $\bm V = \bm H^1_0(\Omega)$ and $Q = L^2_\Delta(\Omega)$. The 
bilinear form $a: \bm H^1_0(\Omega) \times \bm H^1_0(\Omega) \to \mathbb{R}$ in \eqref{eq:mixed_problem} associated with \eqref{eq:pstokes} is continuous and coercive \cite{guermond_volii}. 
The bilinear form $b: \bm H^1_0(\Omega) \times L^2_\Delta(\Omega) \to \mathbb{R}$ is also continuous. 
Thus, conditions \eqref{eq:problem_formulation:a_coercive} - \eqref{eq:problem_formulation:b_continuous} in Theorem \ref{theorem:babuska_brezzi} hold. 
We use the pressure gradient estimate from 
Theorem \ref{theorem:finalresult_necas_extended}, definition of 
the $H^{-1}(\Omega)$ norm, and then divide the resulting estimate by $C_\Delta^{-1}\, C_1$ to arrive for $p\in  L^2_\Delta(\Omega)$ at:
$$
   \sup_{\bm v \in \bm H_0^{1}} \frac{(\nabla p, \bm v)_{H^{-1}(\Omega), H^1_0(\Omega)}}{\| \bm v \|_{H_0^1}} = 
   \sup_{\bm v \in \bm H_0^{1}} \frac{-(p, \nabla \cdot \bm v)}{\| \bm v \|_{H_0^1}} \geq 
   C_\Delta\, C_1^{-1}\, \| p \|_{L^2(\Omega)}.
$$
This is the inf-sup condition \eqref{eq:problem_formulation:b_infsup} with the inf-sup constant $\gamma:=C_\Delta\, C_1^{-1}>0$. 
Thus, condition \eqref{eq:problem_formulation:b_infsup} is also established and \eqref{eq:pstokes} is well-posed with $\bm V = \bm H^1_0(\Omega)$ and $Q = L^2_\Delta(\Omega)$. 
Using the inf-sup condition above we have: 
$$
  \| p \|_{L^2(\Omega)}  \leq \gamma\, \sup_{\bm v \in \bm H_0^{1}} \frac{-(p, \nabla \cdot \bm v)}{\| \bm v \|_{H_0^1}}.
$$
We now substitute the $-(p, \nabla\cdot \bm v)$ term by using the Stokes problem \eqref{eq:pstokes} and make bounds to arrive at:
$$
   \| p \|_{L^2(\Omega)}  \leq \gamma \sup_{\bm v \in \bm H_0^{1}} \frac{- (\nabla \bm u, \nabla \bm v) + (\bm f, \bm v)}{\| \bm v \|_{H_0^1}}.
$$
Next, using the Cauchy-Schwarz inequality on all the numerator terms, and then 
$\|\nabla \bm v\|_{L^2(\Omega)} \leq \|\bm v\|_{\bm H^1_0(\Omega)}$ (and same for the $\bm u$ term)  we have:
\begin{equation}
    \label{theorem:pressure_wellposedness_equation_tmp1}
   \| p \|_{L^2(\Omega)}  \leq  \gamma \left( \|\bm u\|_{\bm H^1_0(\Omega)} + \|\bm f\|_{L^2(\Omega)} \right).
\end{equation}
Inserting $\bm v = \bm u$ to the Stokes problem \eqref{eq:pstokes}, using the incompressibility condition $\nabla \cdot \bm u = 0$ 
gives: $\|\nabla\bm u\|_{L^2(\Omega)}^2 =(\bm f,\bm u) \leq \|\bm f\|_{H^{-1}(\Omega)}\, \|\bm u\|_{\bm H^1_0 (\Omega)}$. 
With the Poincaré's inequality we arrive at $\|\bm u\|_{\bm H^1_0(\Omega)} \leq C_p\|\bm f\|_{H^{-1}(\Omega)}$. The latter relation proves the well-posedness estimate for the velocity vector field. 
Then, inserting the relation to \eqref{theorem:pressure_wellposedness_equation_tmp1} proves the well-posedness estimate for the pressure function.
\end{proof}

We now discuss the incompressibility (divergence-free) condition $\nabla \cdot \bm u = 0$ when solving the Stokes problem \eqref{eq:pstokes} 
under the boundary conditions \eqref{eq:bc_u_p_homogenous_reformulated}. Ideally, this condition holds pointwise, but this is generally 
not true when $\bm u \in \bm V = \bm H^1_0(\Omega)$. 

A key measure of incompressibility is the absolute average divergence, given by
 $|\int_\Omega \nabla \cdot \bm u|$.
We have  $|\int_\Omega \nabla \cdot \bm u\, d\Omega| = |\int_{\partial\Omega} \bm u \cdot \bm n\, ds| = 0,$
as long as $\bm u|_{\partial\Omega} = 0$, which is our assumption. This result is optimal for any choice of $|\Delta|>0$.

Key measures of divergence include its $L^1$- and $L^2$-norm. In the following Lemma \ref{lemma:divfree:divu_L2_DeltaOnly}, 
we establish that usually $\|\nabla \cdot \bm u\|_{L^2(\Omega)}> 0$. Building on this result, it follows that
$\|\nabla \cdot \bm u\|_{L^1(\Omega)} \le C |\Delta|^{1/2}$.

\begin{lemma}[Velocity divergence $L^1$- and $L^2$-norm]
    \label{lemma:divfree:divu_L2_DeltaOnly}
    Let $\bm u \in \bm H^1_0(\Omega)$ and $p \in L^2_\Delta(\Omega)$ be the solutions to the Stokes problem \eqref{eq:pstokes} over $\Omega$. Then 
    $\|\nabla \cdot \bm u\|_{L^2(\Omega\backslash\Delta)} = 0,$ and its equivalence:
\begin{equation}
    \label{eq:divfree:tmp1}
      \|\nabla \cdot \bm u\|_{L^2(\Omega)} = \|\nabla \cdot \bm u\|_{L^2(\Delta)}.
 \end{equation}   
    An a priori estimate of the velocity divergence in $L^1$-norm is:
\begin{equation}
    \label{eq:divfree:tmp2}
    \|\nabla \cdot \bm u\|_{L^1(\Omega)} \leq C_p |\Delta|^{1/2} \|\bm f\|_{H^{-1}(\Omega)}.
 \end{equation}   
\end{lemma}

\begin{proof}
\emph{(i) Divergence $L^2$-norm.}
We utilize the operator $P$ from Definition \ref{definition:operator_P} 
to pro\-ject $\nabla \cdot \bm u$ onto $L^2_\Delta(\Omega)$. 
Since $P(\nabla \cdot \bm u) \in L^2_\Delta(\Omega)$, we can use this as a possible test function in \eqref{eq:pstokes}. This  implies:
$$
\|\nabla \cdot \bm u\|_{L^2(\Omega)}^2 = (\nabla \cdot \bm u, \nabla \cdot \bm u) =  (\nabla \cdot \bm u, \nabla \cdot \bm u - P(\nabla \cdot \bm u)).
$$
With  $P(\nabla \cdot \bm u)|_\Delta=0$ on $\Delta$, and   $P(\nabla \cdot \bm u)|_\Delta=\nabla\cdot\bm u$ on $\Omega\setminus\Delta$ we obtain (\ref{eq:divfree:tmp1})
and $\|\nabla \cdot \bm u\|_{L^2(\Omega\backslash\Delta)} = 0$. \\[2mm]
\emph{(ii) Divergence $L^1$-norm.}
With (\ref{eq:divfree:tmp1}) follows
$$
\|\nabla \cdot \bm u\|_{L^1(\Omega\backslash\Delta)} \leq |\Omega\backslash\Delta|^{1/2}\, \|\nabla \cdot \bm u\|_{L^2(\Omega\backslash\Delta)} = 0.
$$
We further deduce 
$$
\|\nabla \cdot \bm u\|_{L^1(\Omega)} = \|\nabla \cdot \bm u\|_{L^1(\Delta)}  \leq |\Delta|^{1/2}\, \|\nabla \cdot \bm u\|_{L^2(\Delta)}.
$$
Finally, the bound (\ref{eq:divfree:tmp2}) follows with the stability of the velocity gradient (Theorem \ref{theorem:stokes_wellposedness_final}).
\end{proof}

\section{Numerical experiments}\label{section:numerical_experiments}

We present numerical experiments that validate our theoretical results on the well-posedness of the Stokes problem, as stated in Theorem \ref{theorem:stokes_wellposedness_final}.
 In all experiments, we discretize the Stokes problem \eqref{eq:pstokes} using the finite element method and perform computations with the open-source finite element library 
 FEniCS 2019 \cite{UFL,Alnaes2014,Alnaes2015}.
The velocity space  $\bm{V_h}$  and pressure space $Q_h$
are defined in Section \ref{sec:preliminaries}. All numerical simulations were conducted on a laptop equipped with an 
AMD Ryzen 7 PRO 6850U processor and 16 GB RAM.

\subsection{Constant $C_\Delta(|\Delta|)$ under $|\Delta|$-refinement when $p \in L^2_{\Delta}(\Omega)$}

In Lemma \ref{lemma:T_range_closed} we established that $T$ is bounded from below, i.e., that there exists 
$C_\Delta>0$ such that $\| Tp \|_{L^2(\Omega)} \geq C_\Delta\, \|p\|_{L^2(\Omega)}$ for each $p \in L^2_\Delta(\Omega)$. 
Furthermore, we showed that $C_\Delta$ has a lower bound $C_\Delta\sim|\Delta|^{1/2}$ which is the worst case decay as the volume $|\Delta|\to 0$. The largest possible constant $C_\Delta$ is given by
$$
   C_\Delta := \inf_{p \in L^2_\Delta(\Omega)} 
   \frac{\|Tp\|_{L^2(\Omega)}}{\|p\|_{L^2(\Omega)}}.
$$
In this section, we study the dependence of $C_\Delta$ on the volume of $|\Delta|$, from a numerical point of view. 
We compute $\|Tp\|_{L^2(\Omega)}/\|p\|_{L^2(\Omega)}$ for 
different $p \in L^2_\Delta(\Omega)$ 
as $|\Delta|$ is refined. In particular,
we choose five different functions $p_i=Pq_i$, $i=1,\ldots,5$, where $P$ is the projection operator defined in \eqref{eq:operator_P}, and
$q_i=q_i(x_1,x_2)$, are given by:
\begin{equation}
    \label{eq:experiments:constant:pressure_functions}
\begin{aligned}
    q_1 &\equiv 10^6,\quad q_2 \equiv 10^{-6},\quad q_3 = 7\,e^{-10 (x_1^2+x_2^2)}, \\
    q_4 &= (x_1 - 0.2)^2 + (x_2 - 0.2)^2)^{\frac{3}{2}},\quad q_5 = T \sin(2\pi\, x_1\, x_2).
\end{aligned} 
\end{equation}
\begin{figure}[t]
    \centering
    \begin{tabular}{cccc}
        \vspace{-0cm}\hspace{0.4cm}Large const.& Large const.$^*$ & \hspace{0.1cm}Cubic distance & Sine \vspace{-0.05cm}\\
        \hspace{0.1cm}$(p_1)$ & $(p_1)$ & $(p_4)$ &  $(p_5)$  \vspace{-0.05cm}\\
\hspace{-0.3cm}\includegraphics[width=0.26\linewidth]{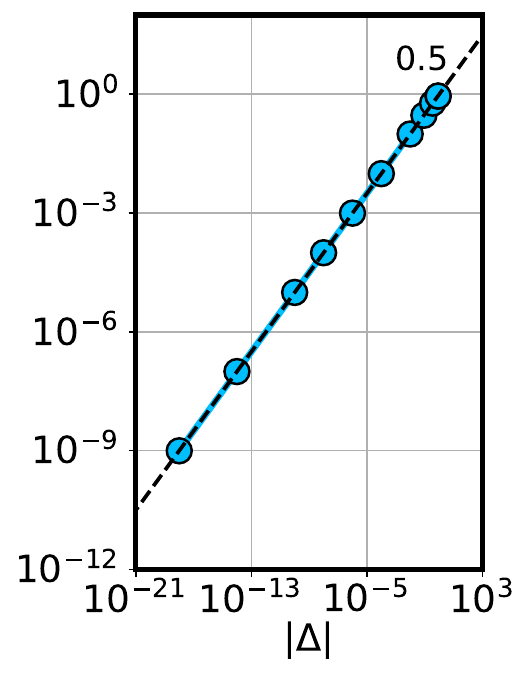} &
    \hspace{-0.65cm}\includegraphics[width=0.26\linewidth]{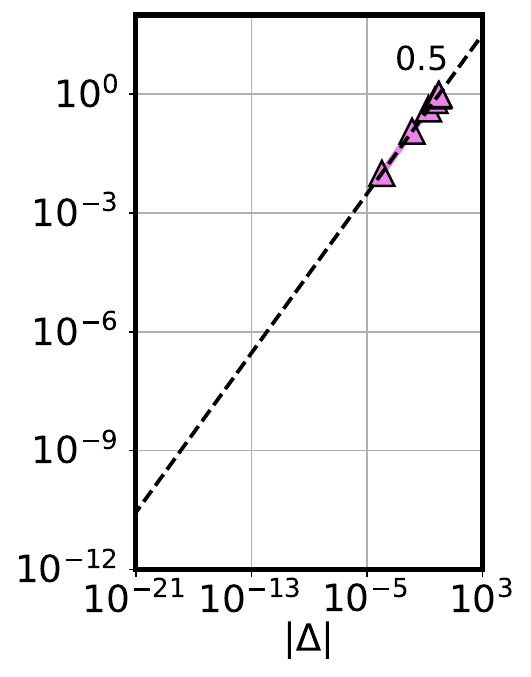} &
    \hspace{-0.32cm}\includegraphics[width=0.25\linewidth]{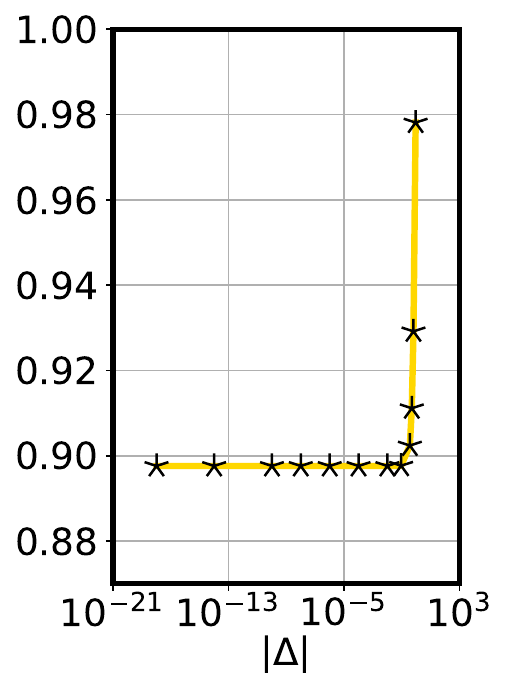} & 
    \hspace{-0.44cm}\includegraphics[width=0.26\linewidth]{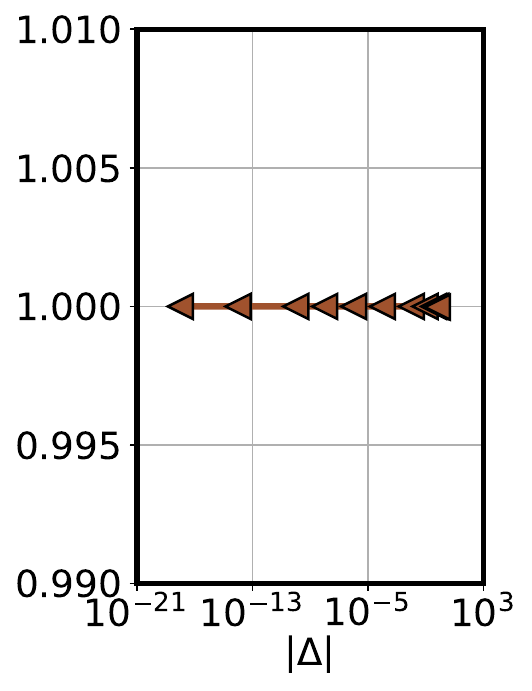} 
    \end{tabular}
    \caption{The value of constant $C = \|Tp\|_{L^2(\Omega)}/\|p\|_{L^2(\Omega)}$ as the volume $|\Delta|$ is refined, for three different example pressure functions \eqref{eq:experiments:constant:pressure_functions}. 
    In plots 1, 3, and 4 (from left to right) $\Delta \subset \Omega$ is a disk centred within a unit disk $\Omega$. 
    In plot 2 (from left to right) $\Delta \subset \Omega$ is a ring around $\partial\Omega$.}
    \label{fig:experiments:constant:constants}
\end{figure}

The key question concerns the decay rate of $C_\Delta$ as $|\Delta| \to 0$, with slower decay rates being preferable.
In most cases, we define $\Delta$ as a small disk inside $\Omega$, allowing us to compute function norms for extremely small $|\Delta|$  (e.g., $|\Delta| \approx 10^{-20}$). 
In one case, however, $\Delta$ is chosen as a ring around $\partial\Omega$, where the smallest considered volume is $|\Delta| \approx 10^{-5}$.

The results, shown in Figure \ref{fig:experiments:constant:constants}, indicate that $C_\Delta$ decays at a rate of $0.5$ for $p_1$
in both settings: when $\Delta$ is a small disk within $\Omega$ and when it forms a ring around $\partial\Omega$. 
This aligns with our worst-case estimate for $C_\Delta$ in Lemma \ref{lemma:T_range_closed}. A similar behaviour is observed for $p_2$. 
For $p_3$, $p_4$ and $p_5$, however, $C_\Delta$ 
exhibits stagnating decay as $|\Delta| \to 0$. Equivalent results for a one-dimensional $\Omega$ confirm these findings.

Both theoretical and numerical observations suggest that:
(i) as $|\Delta|\to 0$, $C_\Delta$ decays to $0$ at most at order $0.5$, 
(ii)  $C_\Delta$ is independent of the dimension.

\subsection{Condition numbers of the stiffness matrix under $|\Delta|$ and mesh refinement}

In this section, we analyze the condition numbers of the stiffness matrix when discretizing \eqref{eq:pstokes} using finite element methods.
We consider three cases where the pressure solution function is constrained to:
(i) have zero mean on $\Omega$, 
(ii) vanish over $\partial \Omega$, 
(iii) vanish over $\Delta \subset \Omega$. 
For $|\Delta|$-refinement, we define $\Delta$ as a circular region centred at $(1,0)$ within the unit disc. 
Under mesh refinement, $\Delta$ is instead chosen as a ring encompassing the triangles nearest to the disk boundary. 
The results are presented in Figure \ref{fig:experiments:condition_numbers}. 

\begin{figure}[t]
    \centering
    \begin{tabular}{ccc}
    \hspace{0.85cm}Mesh refinement & \hspace{1cm}$\Delta$-refinement \\
    \raisebox{0.14cm}{\includegraphics[width=0.32\linewidth]{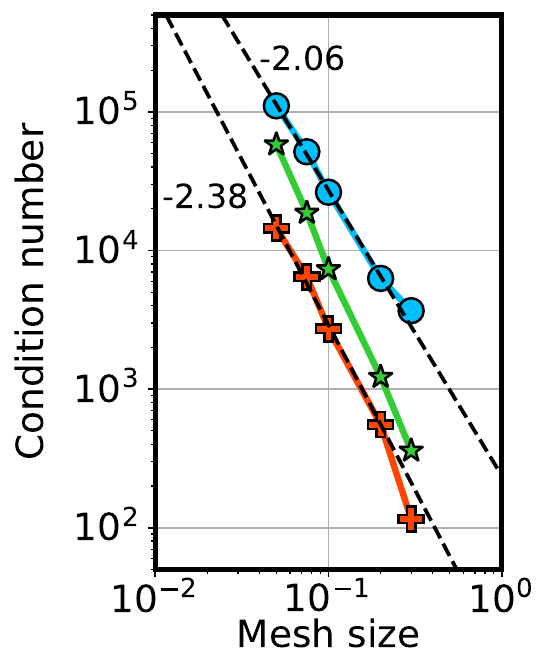}} &
    \includegraphics[width=0.3\linewidth]{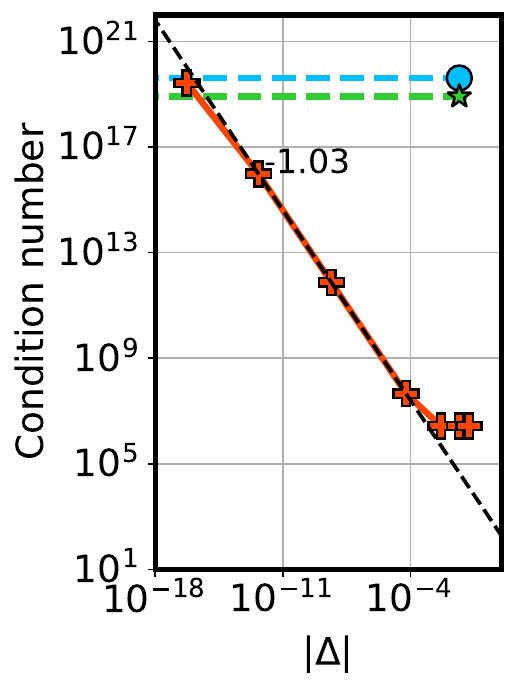} &
    \raisebox{3cm}{\includegraphics[width=0.2\linewidth]{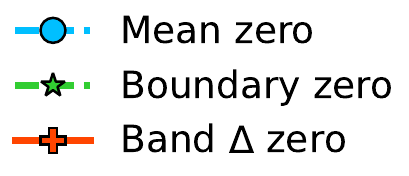}}
    \end{tabular}
    \caption{Condition numbers of the stiffness matrix, when the pressure solution function is constrained to: (i) have zero mean on $\Omega$ (label Mean zero), (ii) vanish over $\partial \Omega$ (label Boundary zero), (iii) vanish over $\Delta \subset \Omega$ (label Band $\Delta$ zero). 
    The condition numbers are plotted as a function of volume of $\Delta\subset \Omega$ when mesh size is fixed (right), 
    and as a function of mesh size (left). 
    In the right plot $\Delta$ is an area centred around a point $(1,0)$. In the left plot $\Delta$ is a ring containing triangles closest to domain boundary $\partial\Omega$.}
    \label{fig:experiments:condition_numbers}
\end{figure}

We observe that under $|\Delta|$-refinement, the condition number grows linearly (order 1) in case (iii). When  $|\Delta|$ is very small, its condition number is comparable to those in cases (i) and (ii).
Under mesh refinement, the condition number grows quadratically (order 2) in case (i), as expected. However, in case (iii), it grows at order 2.4, while in case (ii), it increases at order 2.8. Across the considered mesh sizes, case (iii) yields the smallest condition numbers, whereas case (i) results in the largest.
Overall, the condition numbers for cases (i) and (iii) are comparable. The exception is case (ii), where enforcing pressure vanishing on $\partial\Omega$
leads to significantly higher growth rates.

\subsection{Model coupling using the modified Dirichlet boundary conditions}

In this section, we numerically investigate the coupling between the Stokes model and an approximate fluid flow model. The two models share a domain interface and are connected through reformulated Dirichlet boundary conditions for velocity and pressure \eqref{eq:bc_u_p_homogenous_reformulated}.

The objective is to solve the original Stokes problem \eqref{eq:pstokes} over the subdomain
$\Omega_{\text{sub}} \subset \Omega$, while 
using $\bm \delta$-perturbed (inexact) Dirichlet boundary data. This data is obtained from the solution of a perturbed Stokes problem over the complementary subdomain $\Omega\setminus\Omega_{\text{sub}}$.
The subdomain setup is illustrated in Figure \ref{fig:experiments:domains_delta_coupling_setup}.
This setting models various fluid flow coupling problems, such as those in 
\cite{ISCALgrounding,ISSM_iscal,kuchta_stokes_darcy}.
We numerically verify the well-posedness estimate from Theorem \ref{theorem:stokes_wellposedness_final}, focusing on both mesh $(h)$ refinement and 
$|\Delta|$-refinement. Additionally, we explore numerical stability under different choices of $|\Delta|$ and $\bm \delta$. In some cases, we use the diameter of the $\Delta$
region, $D(\Omega)$, to compare stability when the pressure boundary condition is imposed on $\partial\Omega$ $(D(\Omega)=0)$ versus when it is imposed on $\Delta\subset\Omega$
with $|\Delta|>0$  $(D(\Delta)>0)$.

\begin{figure}[b]
    \centering
    \begin{tabular}{cc}
        \includegraphics[width=0.33\linewidth]{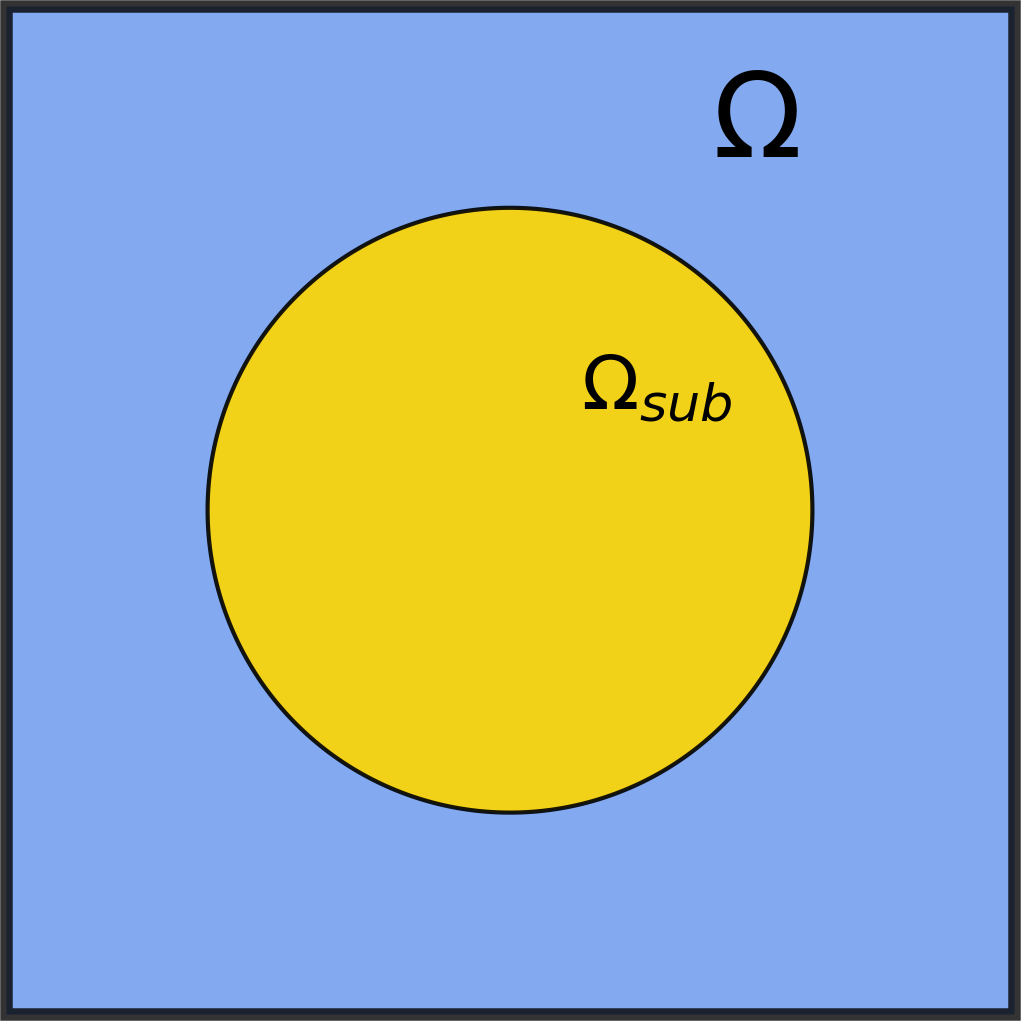} &
        \includegraphics[width=0.33\linewidth]{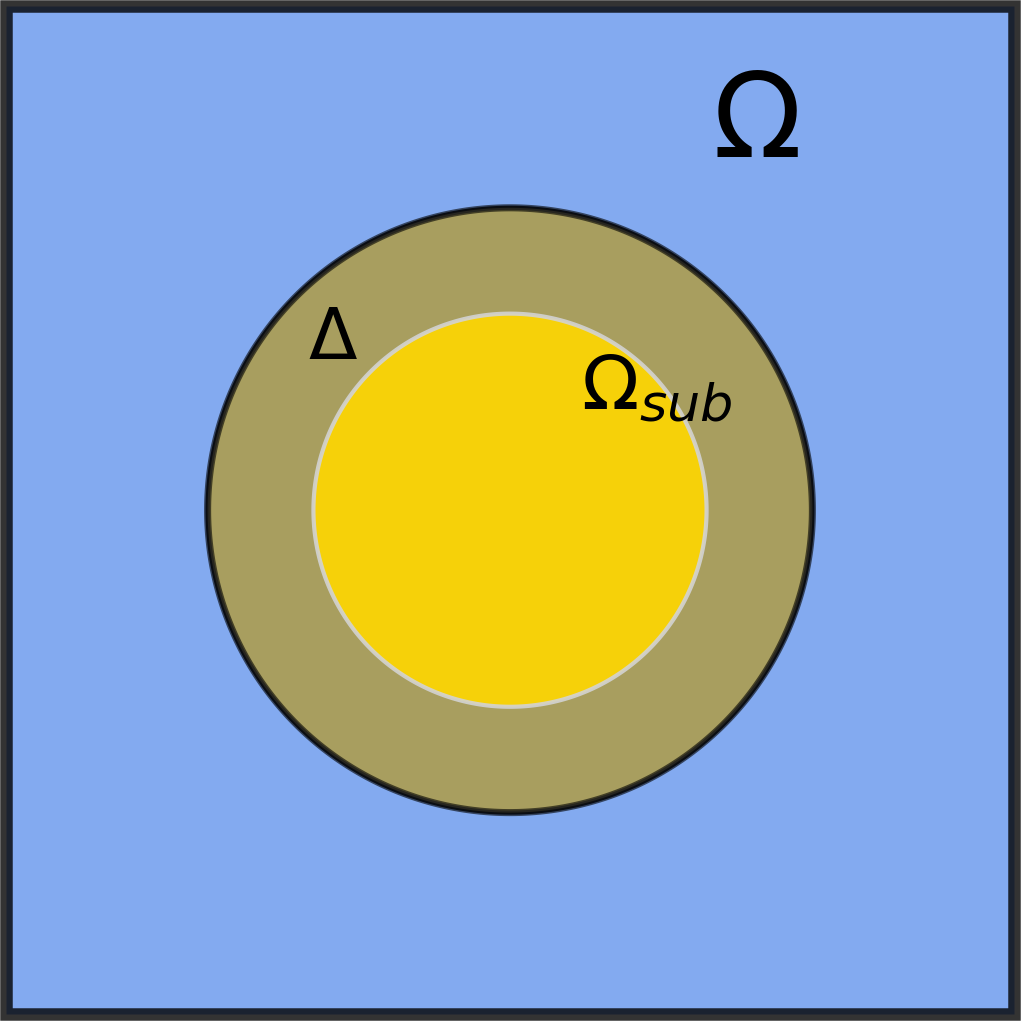}
    \end{tabular}
    \caption{The domain subdivision used for coupling an exact Stokes problem over $\Omega_{\text{sub}}\subset\Omega$ 
    to a perturbed Stokes problem over $\Omega\backslash\Omega_{\text{sub}}$, 
    where $\Delta\subset\Omega_{\text{sub}}$ is a small band where Dirichlet boundary conditions are imposed. Left: $D(\Delta) = 0, |\Delta|=0$. Right: $D(\Delta) > 0, |\Delta|>0$.}
    \label{fig:experiments:domains_delta_coupling_setup}
\end{figure}

\subsubsection{Model coupling: Configuration}

The configuration is as follows. We first solve the Stokes problem \eqref{eq:pstokes} over a square $\Omega = (0,1)^2$ with $\bm f = (0, -1)$ 
 to get the solution pair
$(\bm u,p)$. 
After that, we solve the Stokes problem 
for $(\bm{u_{\text{sub}}}, p_{\text{sub}})$ over the circle 
$\Omega_{\text{sub}} = \{ \bm x \in \Omega\, |\, |\bm x-\bm x_0|<r \},$
 with radius $r:=0.2$ centered in the middle $\bm x_0$ of the square $\Omega$. To obtain 
$(\bm{u_{\text{sub}}}, p_{\text{sub}})$ over $\Omega_{\text{sub}}$ we impose boundary conditions 
$    \bm{u_{\text{sub}}} |_{\partial\Omega_{\text{sub}}} = \bm u|_{\partial\Omega_{\text{sub}}} + \bm{u_\delta}|_{\partial\Omega_{\text{sub}}}$ 
and 
$p_{\text{sub}} |_{\Delta} =  p|_{\Delta} + p_\delta|_{\Delta}$.
Here $\bm{u_\delta}|_{\partial\Omega_{\text{sub}}}$ and $p_\delta|_{\Delta}$ are perturbation functions in velocity (restricted to $\partial\Omega_{\text{sub}}$) and pressure (restricted to $\Delta_{\text{sub}}$) 
respectively. We define perturbation functions over $\Omega$ as:
$(u_\delta)_1 = (u_\delta)_2 = p_\delta = \delta\phi \|\phi\|_{L^\infty(\Omega_{sub})}^{-1}$, where  $\phi(x_1,x_2):=\cos (4 \pi\, x_1\, x_2)$ for
$(x_1,x_2) \in \Omega$,
and $\delta>0$ is a relative measure of the extent to which the problem is perturbed. Case $\delta = 0$ corresponds to no perturbation, 
whereas $\delta = 1$ corresponds to full perturbation. The perturbation function represents the model approximation error in an approximate fluid flow model. 
In addition, we also construct a perturbed solution over the whole $\Omega$ as 
$\bm{\tilde u} = \bm u + \bm{u_\delta}$ and $\tilde p  =  p + p_\delta$.
In the end we merge $(\bm{\tilde u}, \tilde p)$ with $(\bm{u_{\text{sub}}}, p_{\text{sub}})$ as follows:
\begin{equation}
    (\bm u_{\text{merged}}, p_{\text{merged}}) = 
    \begin{cases} 
        (\bm{\tilde u}, \tilde p), & \text{ on } \Omega \backslash \bar\Omega_{\text{sub}} \\
        (\bm{u_{\text{sub}}}, p_{\text{sub}}), & \text{ on } \bar \Omega_{\text{sub}}.
     \end{cases}
     \label{eq:numerical_experiments:stokes_coupled_solution_merged}
    \end{equation}
The merged solution mimics a solution of two coupled problems, where the solution over $\Omega \backslash \Omega_{\text{sub}}$ comes from a perturbed (inexact) Stokes problem, and where 
the solution over $\Omega_{\text{sub}}$ comes from a full (exact) Stokes problem but with perturbed boundary conditions.

\subsubsection{Model coupling: solutions}

\begin{figure}[t]
    \centering
    \begin{tabular}{ccc}
        \hspace{-1cm}Velocity $u_1$ & \hspace{-1.3cm}Velocity $u_2$ & \hspace{-1.3cm}Pressure $p$ \\
        \hspace{-0cm}\includegraphics[width=0.32\linewidth]{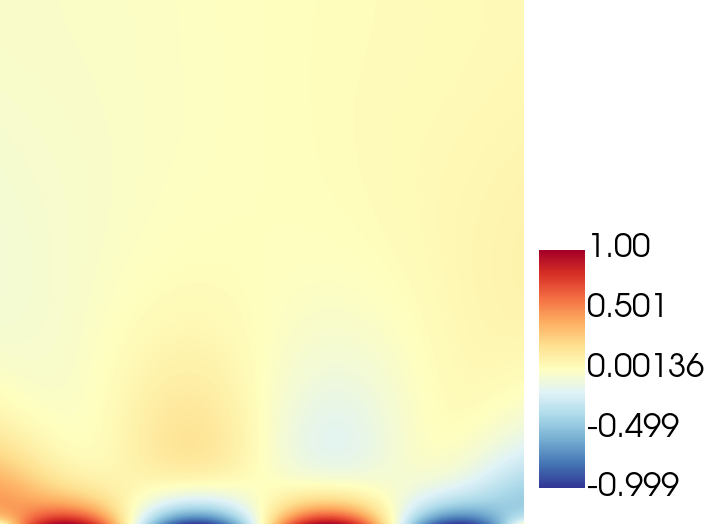} &
        \hspace{-0.35cm}\includegraphics[width=0.315\linewidth]{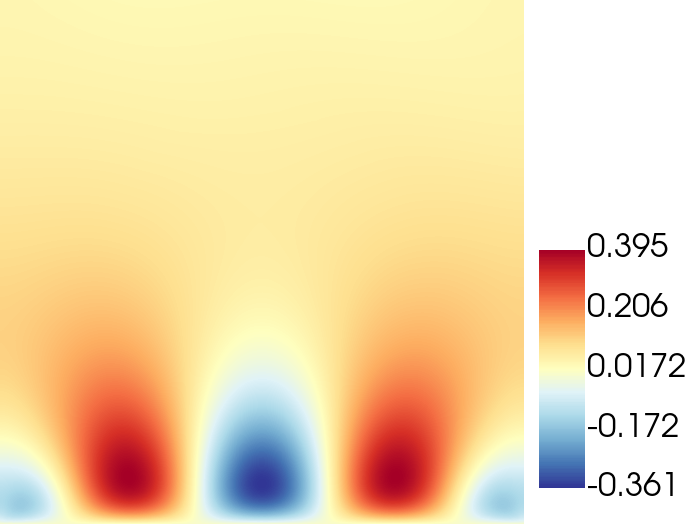} &
        \hspace{-0.35cm}\includegraphics[width=0.305\linewidth]{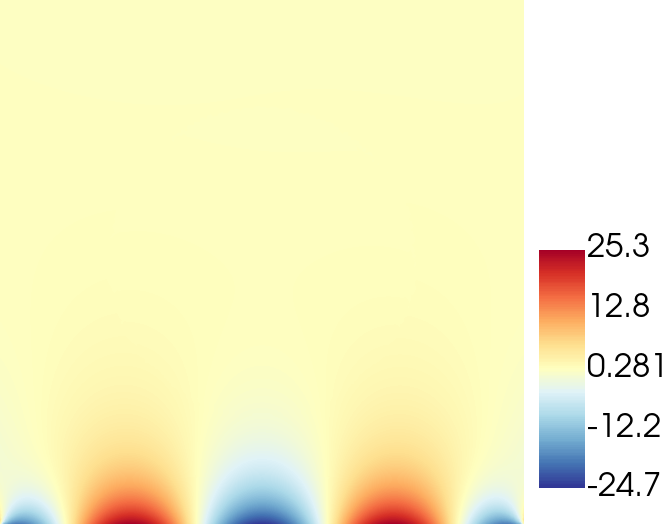} \vspace{0.1cm}
        \\
         \includegraphics[width=0.312\linewidth]{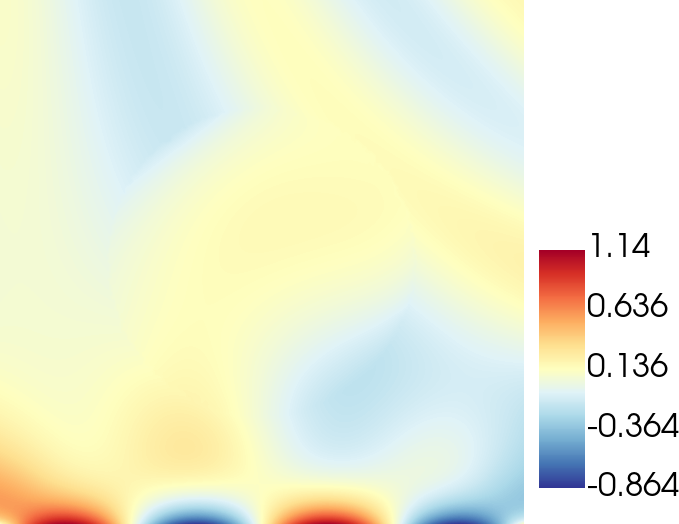} &
        \hspace{-0.3cm}\includegraphics[width=0.32\linewidth]{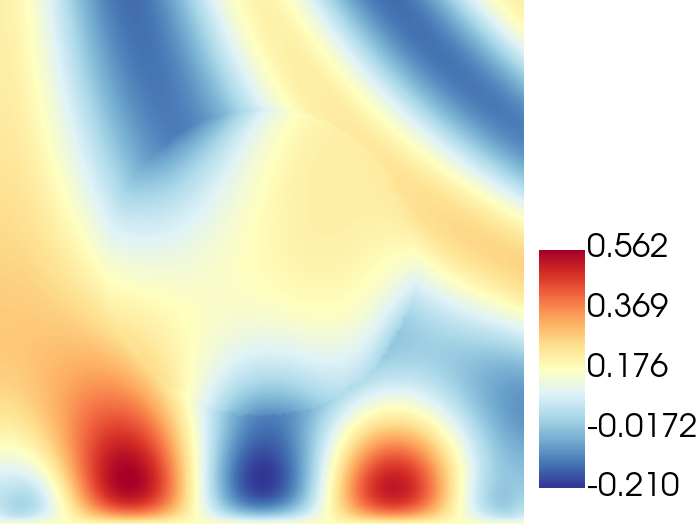} &
        \hspace{-0.3cm}\includegraphics[width=0.308\linewidth]{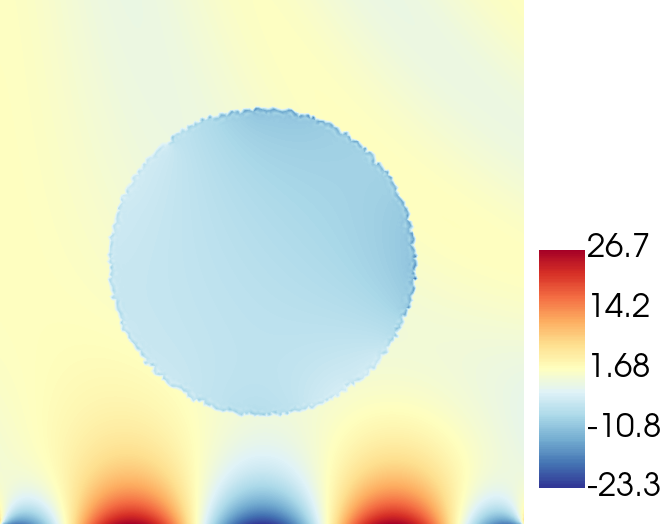} \vspace{0.1cm}\\
        \includegraphics[width=0.312\linewidth]{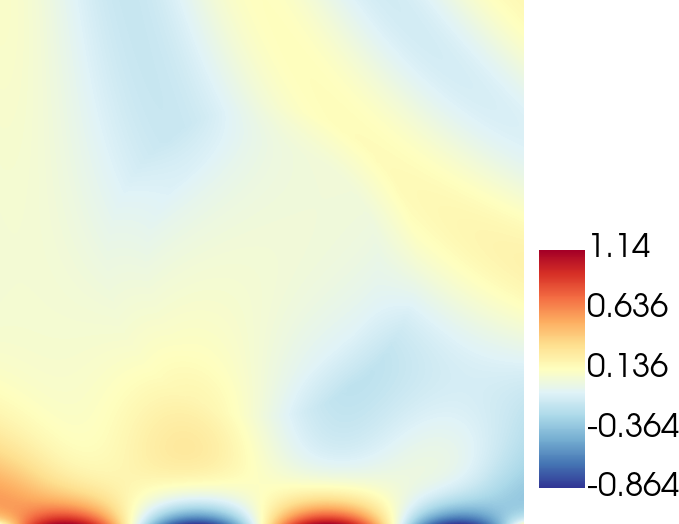} &
        \hspace{-0.3cm}\includegraphics[width=0.312\linewidth]{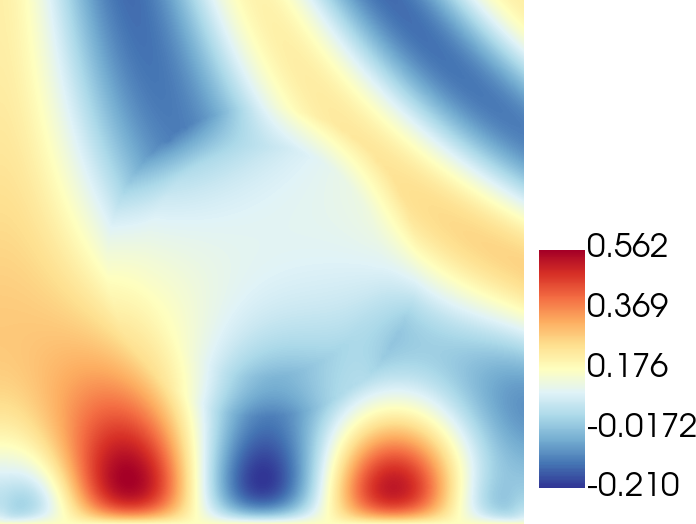} &
        \hspace{-0.3cm}\includegraphics[width=0.312\linewidth]{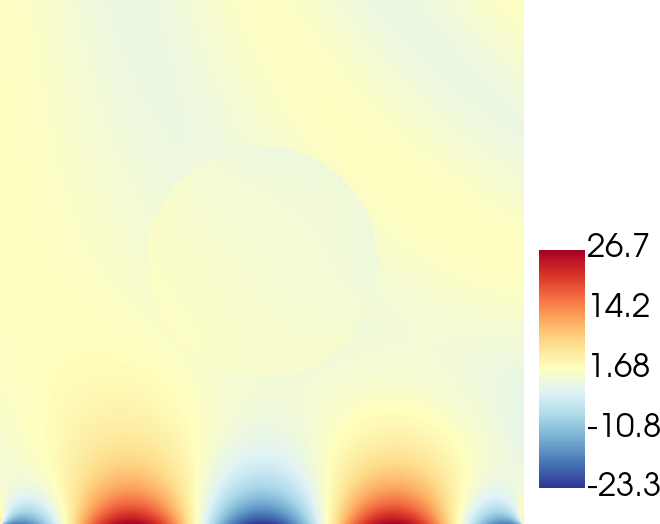} \vspace{0.1cm}
    \end{tabular}
    \caption{Solutions to the perturbed coupled Stokes problem with pressure boundary conditions 
    of the subproblem enforced over $\partial\Omega_{\text{sub}}$.   
    Plots from left to right are: horizontal velocity $u_1$, vertical velocity $u_2$, pressure $p$.
    Upper row: ring diameter $D(\Delta) = 0$, no perturbation $\delta=0.01$.
    Middle row: ring diameter $D(\Delta) = 0$, large perturbation $\delta=1$.
    Lower row: ring diameter $D(\Delta) = 8h$, large perturbation $\delta=1$, and mesh size $h=0.007$.}
    \label{fig:experiments:coupling_subdomains:merged_solution_large perturbation}
    \vspace*{2mm}
    \centering
    \begin{tabular}{ccc}
        \hspace{-0.9cm}Velocity $u_1$ & \hspace{-1.3cm}Velocity $u_2$ & \hspace{-1.2cm}Pressure $p$ \\
        \includegraphics[width=0.305\linewidth]{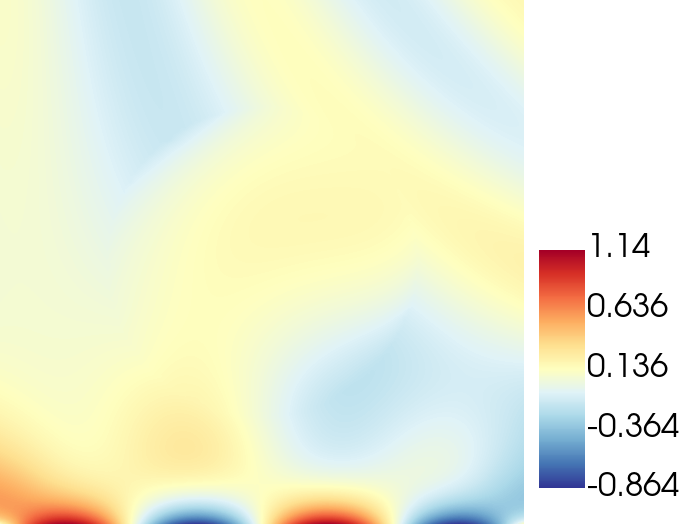} &
        \hspace{-0.3cm}\includegraphics[width=0.312\linewidth]{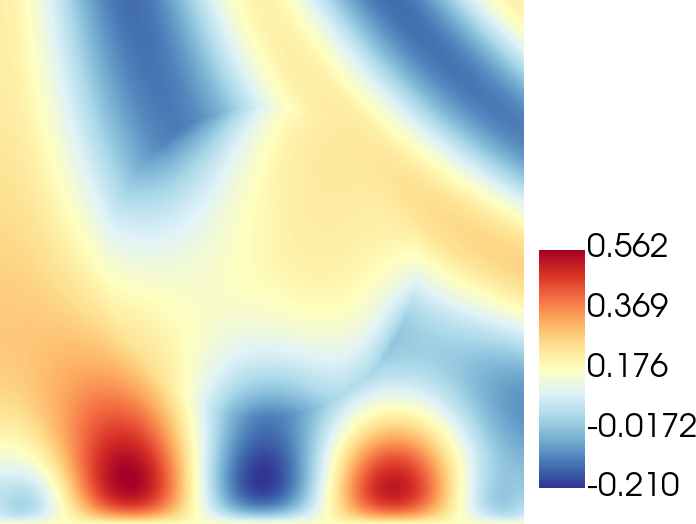} &
        \hspace{-0.3cm}\includegraphics[width=0.295\linewidth]{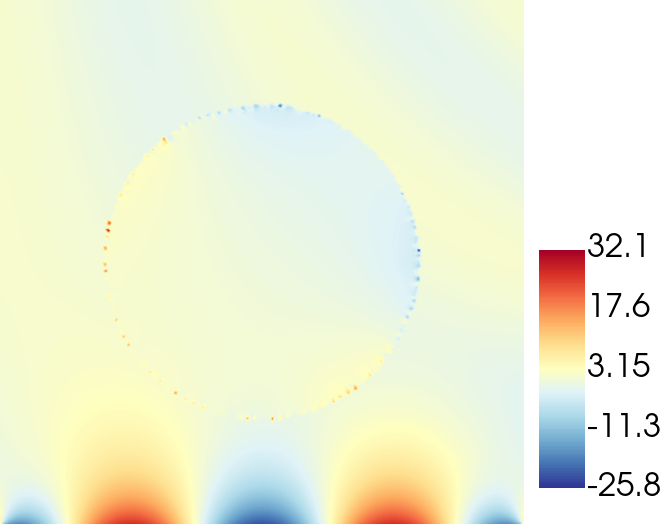} \vspace{0.1cm}\\
    \end{tabular}
    \caption{Solutions to the perturbed coupled Stokes problem with large perturbation $\delta=1$
    and mesh size $h=0.007$. The subproblem is enforced to have zero average $\int_{\Omega_{\text{sub}}} p\, d\Omega = 0$. Plots from left to right are: horizontal velocity $u_1$, vertical velocity $u_2$, pressure $p$.}
    \label{fig:experiments:coupling_subdomains:merged_solution_large perturbation_zero_mean}
\end{figure}

In Figure \ref{fig:experiments:coupling_subdomains:merged_solution_large perturbation} 
we display the coupled solutions $(\bm u_{\text{merged}}, p_{\text{merged}})$ over $\Omega$ as defined in \eqref{eq:numerical_experiments:stokes_coupled_solution_merged}.
The results are shown for subproblem boundary condition perturbation magnitudes 
 $\delta=10^{-2}$ and $\delta=1$. 
When $\delta=10^{-2}$, the coupled solution appears, by visual inspection, identical across all subproblem pressure boundary conditions, whether imposed on
$\partial\Omega_{\text{sub}}$ ($\Delta = \emptyset$) or $\Delta\subset\Omega$ ($\Delta = 8h$). 
However, this consistency breaks down when $\delta=1$, where we observe:
(i) $(\bm u_{\text{merged}}, p_{\text{merged}})$  differ significantly from the original $(\bm u_, p)$ for both choices of $\Delta$, 
(ii) A pronounced boundary layer forms at the interface between $\Omega \backslash \Omega_{\text{sub}}$ and $\Omega_{\text{sub}}$ in all cases.
(iii)  The transition at the interface becomes smoother as  $D(\Delta)$ increases.

In Figure \ref{fig:experiments:coupling_subdomains:merged_solution_large perturbation_zero_mean}, we observe the same behavior when the subproblem solution
is obtained under the constraint  $\int_{\Omega_{\text{sub}}} p\, d\Omega = 0$. 
This suggests that the presence of the boundary layer is caused by the large perturbation in the boundary conditions used to compute $(\bm{u_{\text{sub}}}, p_{\text{sub}})$, rather than by the choice $p \in L^2_\Delta(\Omega_{\text{sub}})$.

\begin{figure}[t]
    \centering
    \begin{tabular}{ccc}
        \hspace{-1.1cm}$D(\Delta) = 0$ & \hspace{-1.2cm}$D(\Delta)=8h$ & \hspace{-1.4cm}$\Delta =\Omega$  \\
        \hspace{-0.1cm}\includegraphics[width=0.313\linewidth]{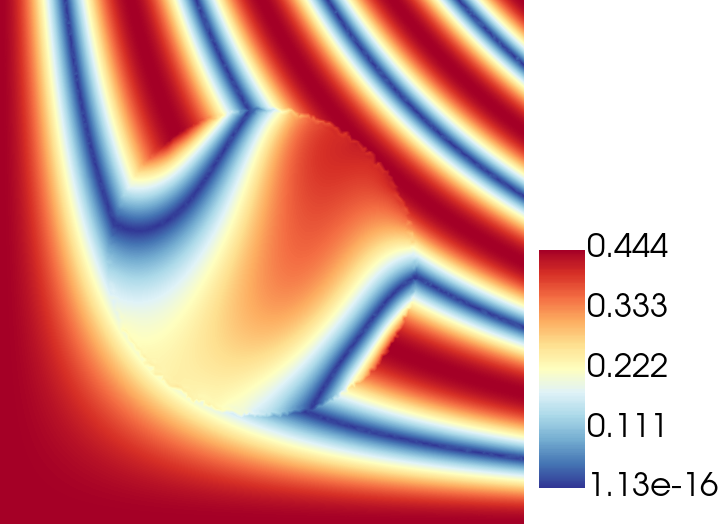} &
        \hspace{-0.2cm}\includegraphics[width=0.313\linewidth]{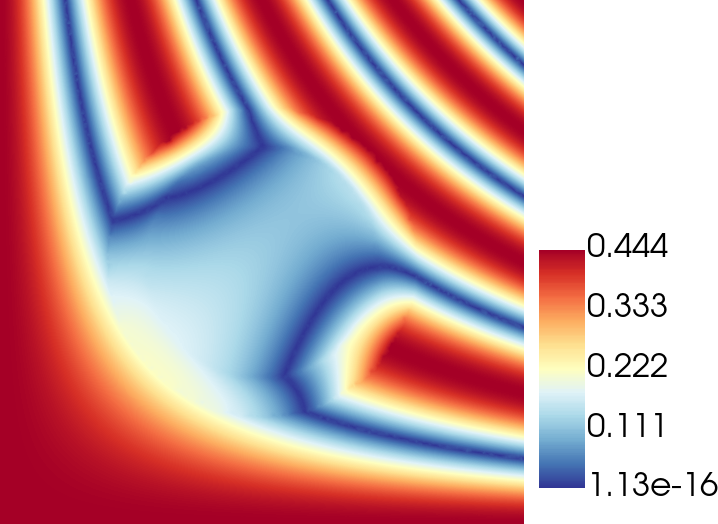} &
        \hspace{-0.2cm}\includegraphics[width=0.313\linewidth]{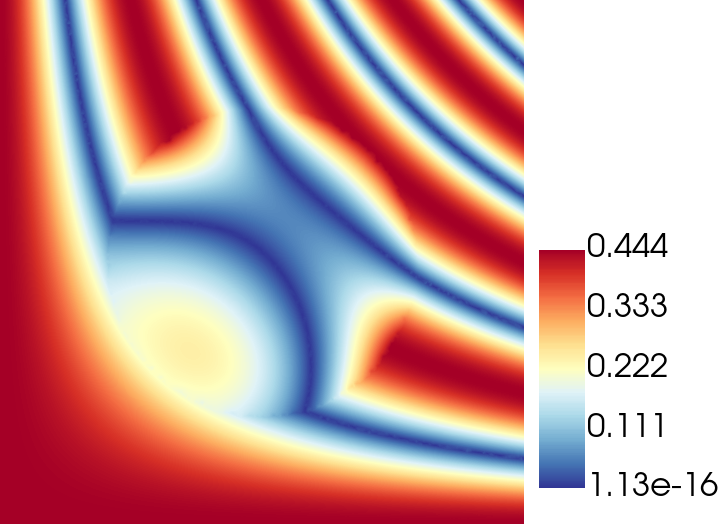} \\
    \end{tabular}
    \caption{
    Spatial 
    error distribution in the vertical velocity solution $u_2$. Perturbed coupled Stokes problem with large perturbation $\delta=1$ and  mesh size $h=0.007$.
    The subproblem pressure boundary conditions are enforced over $\partial\Omega_{\text{sub}}$. From left to right: ring diameter $D(\Delta) = 0$, $D(\Delta)=8h$, $\Delta=\Omega$.}
    \label{fig:experiments:coupling_subdomains:solution_Delta_refinements}
\end{figure}

\subsubsection{Model coupling: spatial relative absolute error distribution}

In Figure \ref{fig:experiments:coupling_subdomains:solution_Delta_refinements}, we present the relative spatial error distribution 
$$
e = |u_{2_{\text{merged}}} - u_2| \|u_2\|_{L^\infty(\Omega)}^{-1},
$$
for the horizontal velocity component $u_{2_{\text{merged}}}$, considering different 
choices of $\Delta$ with a fixed mesh size $h=0.007$. 
We focus exclusively on cases with large boundary condition perturbations.
As $|\Delta|$ increases, we again observe a smoothing effect over 
the boundary layer at the interface between $\Omega \backslash \Omega_{\text{sub}}$ and $\Omega_{\text{sub}}$. 
Additionally, the error within $\Omega_{\text{sub}}$ decreases with increasing $|\Delta|$ (except for the limiting case $\Delta = \Omega$). 
When $D(\Delta)=8h$, the error inside $\Omega_{\text{sub}}$ is significantly smaller than inside $\Omega \backslash \Omega_{\text{sub}}$, 
which is desirable when coupling perturbed models to exact models.
The observed error behaviour is linked to the well-posedness estimate.
According to Theorem \ref{theorem:finalresult_necas_extended}, the well-posedness estimate for pressure depends on $C_\Delta^{-1}$, where $C_\Delta\sim|\Delta|^{1/2}$.
Consequently, increasing $|\Delta|$ decreases $C_\Delta^{-1}$, providing an explanation for the observed error reduction in 
Figure \ref{fig:experiments:coupling_subdomains:solution_Delta_refinements}.

\subsubsection{Model coupling: solution norms under mesh refinement}

In Figure 
\ref{fig:experiments:coupling_subdomains:solution_norms_mesh_refinement_large_perturbation}, we present the norms of the solution under mesh refinement for cases with a large subproblem boundary condition perturbation $\delta = 2$. The left and middle plots display the velocity and pressure norms, respectively, when 
$D(\Delta)$ varies with the mesh size.
 The right plot shows the pressure norm when $D(\Delta)$ is fixed. 
We observe that in the first two plots, the error norms remain nearly constant when the pressure is computed under a zero-mean constraint. This is a desirable outcome, indicating numerical stability across refinements.

\begin{figure}[t]
    \centering
    \begin{tabular}{ccc}
        \hspace{0.6cm}Velocity & \hspace{0.9cm}Pressure & \hspace{0.9cm}Pressure \vspace{-0.05cm}\\
        \includegraphics[width=0.2775\linewidth]{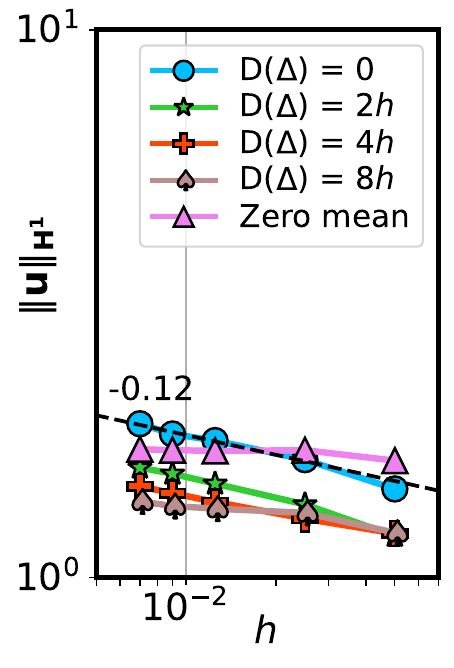} &
        \includegraphics[width=0.3\linewidth]{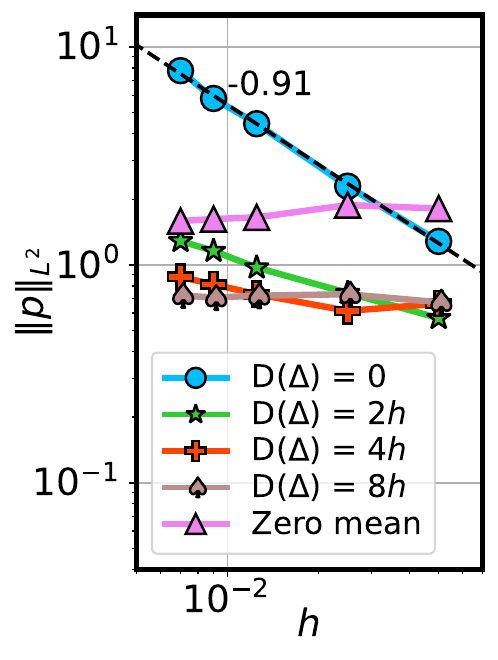}&
        \includegraphics[width=0.3\linewidth]{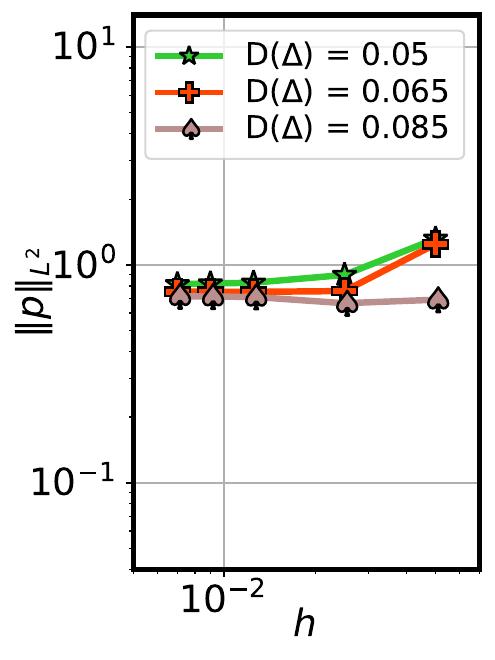}
    \end{tabular}
    \caption{Solution norm convergence under mesh size ($h$) refinement when the perturbation in the Stokes subproblem 
    boundary conditions is large ($\delta = 2$). In the first two plots from left to right we measure 
    the velocity $\bm H^1(\Omega_{\text{sub}})$ norm, 
    and the pressure $L^2(\Omega_{\text{sub}})$ norm when the pressure boundary condition is imposed over different choices of $\Delta \subset \Omega_{\text{sub}}$, where diameter $D(\Delta)$ depends on the mesh size $h$. 
    In the third plot from left to right, we measure the pressure $L^2(\Omega_{\text{sub}})$ norm when $D(\Delta)$ is mesh independent.}
    \label{fig:experiments:coupling_subdomains:solution_norms_mesh_refinement_large_perturbation}
\end{figure}

When pressure is prescribed over $\partial\Omega$, i.e. $D(\Delta) = 0$, the velocity and pressure norms grow at rates of approximately $0.1$ and $1$, respectively, as $h \to 0$. This behaviour is undesirable. However, when $D(\Delta)=kh>0$  is defined as a function of $h$, we observe a stabilizing effect: both velocity and pressure norms flatten as the coefficient $k$ grows. 
Despite this improvement, achieving no growth of the norm for each choice of $D(\Delta) > 0$ remains a key objective. 
To ensure well-behaved pressure norms as $h \to 0$, we set $D(\Delta)>0$ independent of $h$. 
The third (right) plot of Figure \ref{fig:experiments:coupling_subdomains:solution_norms_mesh_refinement_large_perturbation}. 
confirms this approach yields the desired stability.

In Figure \ref{fig:experiments:coupling_subdomains:pressure_convergence_plots_small_perturb}, we repeat the experiment, this time with a small perturbation in the subproblem boundary conditions ($\delta=10^{-2}$). 
In this case, all error norms remain bounded regardless of the choice of $D(\Delta)$, 
 including when the pressure boundary condition is imposed solely on $\partial\Omega$ ($D(\Delta)=0$ case).

\begin{figure}[t!]
    \centering
    \begin{tabular}{ccc}
        \hspace{0.6cm}Velocity & \hspace{0.9cm}Pressure & \hspace{0.9cm}Pressure \vspace{-0.05cm}\\
        \includegraphics[width=0.3\linewidth]{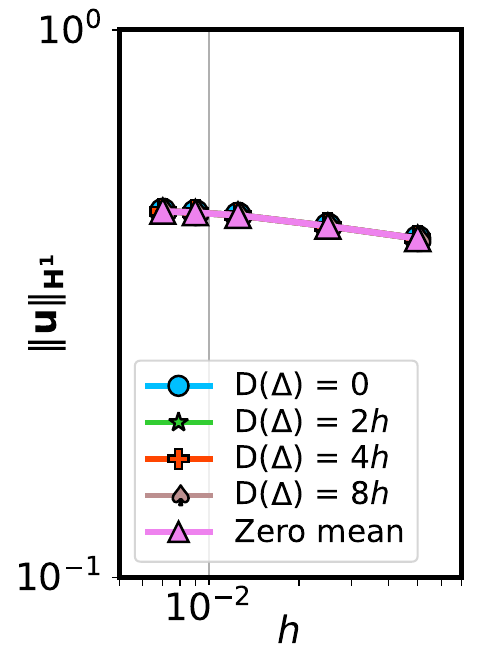} &
        \includegraphics[width=0.3\linewidth]{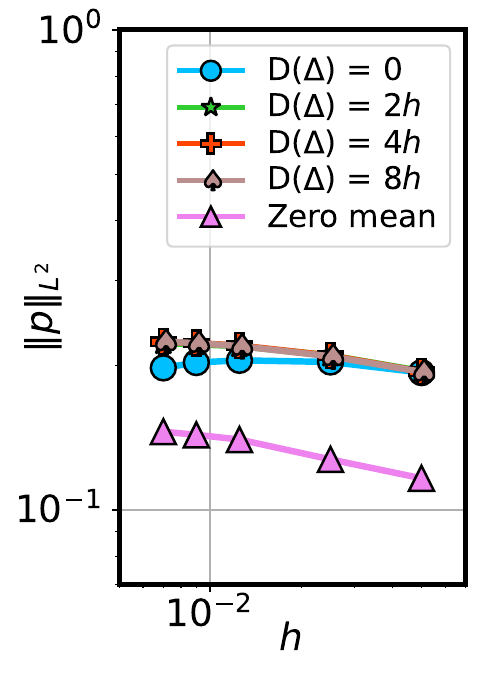}&
        \includegraphics[width=0.3\linewidth]{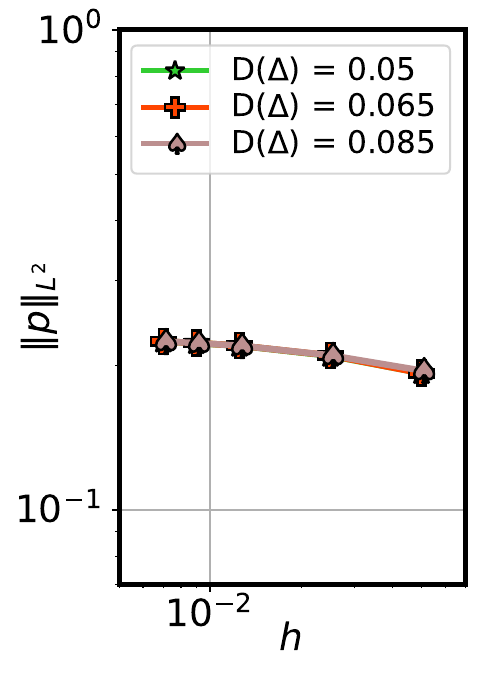}
    \end{tabular}
    \caption{Solution norm convergence under mesh size ($h$) refinement when the perturbation in the Stokes subproblem 
    boundary conditions is small ($\delta = 10^{-2}$). In the left plot, we show
    the velocity $\bm H^1(\Omega_{\text{sub}})$ norm, in the middle plot we display the pressure $L^2(\Omega_{\text{sub}})$ norm when the pressure boundary condition is imposed over different choices of $\Delta \subset \Omega_{\text{sub}}$, with the diameter $D(\Delta)$ depending on the mesh size $h$. 
    In the right plot,  we show the pressure $L^2(\Omega_{\text{sub}})$ norm when $D(\Delta)$ is independent of $h$.}
    \label{fig:experiments:coupling_subdomains:pressure_convergence_plots_small_perturb}
\end{figure}

To avoid instabilities when boundary condition perturbations are large, we recommend setting
$D(\Delta) \propto |\Delta|$, independent of the mesh size, when using the boundary conditions \eqref{eq:bc_u_p_homogenous_reformulated} to solve the Stokes problem.

\subsubsection{Model coupling: solution norms under $|\Delta|$-refinement}

In Figure \ref{fig:experiments:coupling_subdomains:pressure_convergence_plots_large_perturb_Delta_refinement} we 
present the velocity and pressure norms as $|\Delta| \to 0$, while fixing the mesh size to $h=0.007$, for different subproblem boundary condition perturbation $\delta$. 
When  $\delta$ is small, we observe that the velocity and pressure norms remain constant as $|\Delta| \to 0$. 
However, for large perturbations $\delta \geq 1$, the velocity and pressure norms begin to grow with approximate orders of $0.1$ and $0.5$, respectively.  For the largest perturbation considered  ($\delta=16$), the norm growth is bounded, as indicated by the dashed brown line. 
This behaviour can be explained by the well-posedness constant $C_\Delta$ in Lemma \ref{lemma:T_range_closed}. Theoretically, we observe that $(C_\Delta)^{-1}$ grows with order $0.5$, provided that $p \in L^2_\Delta(\Omega_{\text{sub}})$ 
approaches a constant function. As further support, Figure 
\ref{fig:p-delta} shows that the pressure function indeed becomes nearly constant as $|\Delta| \to 0$.

\begin{figure}[t]
    \centering
    \begin{tabular}{ccc}
        \vspace{-0.05cm}\hspace{0.7cm}Velocity & \hspace{0.7cm}Pressure \\
        \includegraphics[width=0.28\linewidth]{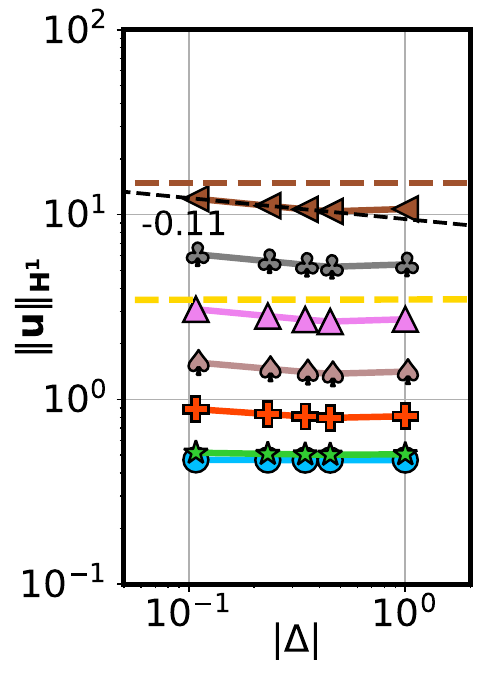}&
        \includegraphics[width=0.28\linewidth]{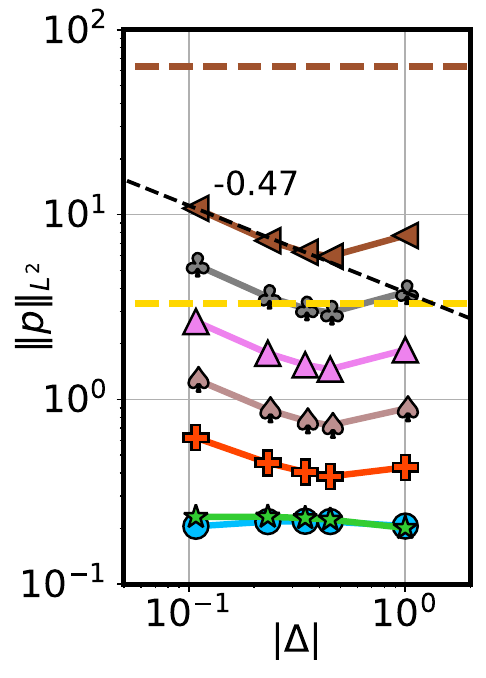} &
        \raisebox{1cm}{\includegraphics[width=0.19\linewidth]{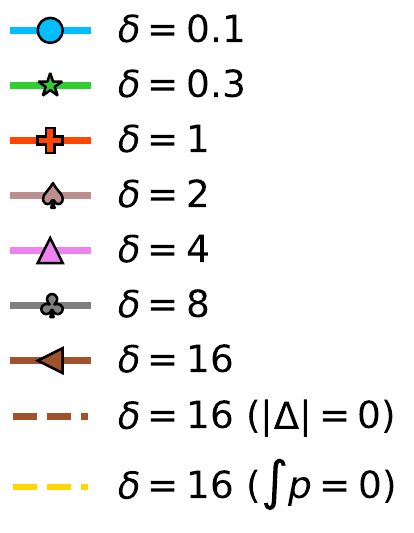}}
    \end{tabular}
    \caption{Solution norm convergence under $|\Delta|$-refinement for different choices of the perturbation magnitudes $\delta$ in the Stokes 
    subproblem boundary conditions, $h=0.007$. 
    Here $|\Delta|$ is the volume of the region $\Delta \subset \Omega_{\text{sub}}$ where the subproblem pressure Dirichlet 
    boundary condition is imposed. The dashed lines represent the limiting cases 
    when $|\Delta|=0$ and $\int_{\Omega_{\text{sub}}} p\, d\Omega = 0$ 
    for the largest perturbation magnitude $\delta = 16$.}
    \label{fig:experiments:coupling_subdomains:pressure_convergence_plots_large_perturb_Delta_refinement}
    \vspace*{4mm}
    \centering
    \begin{tabular}{cccc}
        \hspace{-0.7cm}$|\Delta|=0.44$ & \hspace{-0.7cm}$|\Delta|=0.23$  & \hspace{-0.7cm}$|\Delta|=0.11$ & \hspace{-0.7cm}$|\Delta|=0$  \\
        \includegraphics[width=0.22\linewidth]{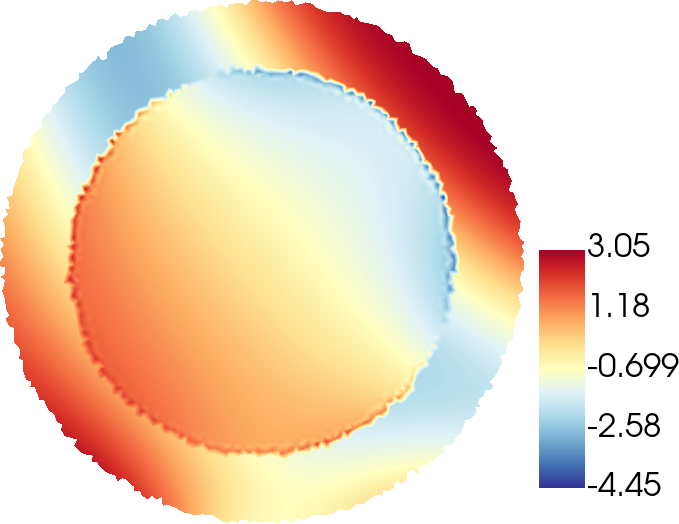} &
        \includegraphics[width=0.22\linewidth]{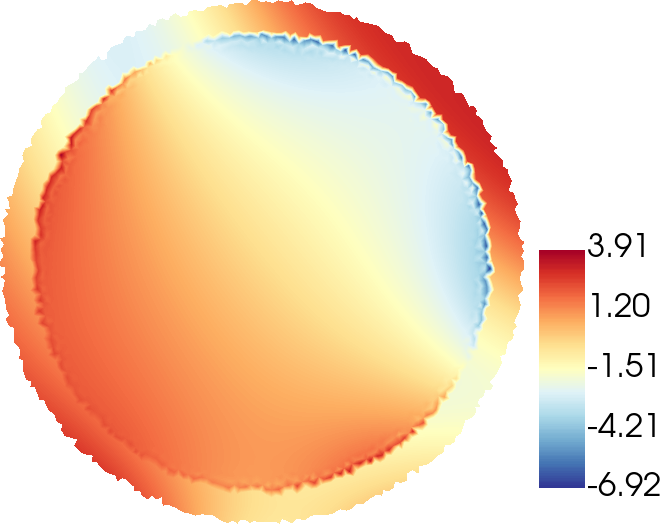} &
        \includegraphics[width=0.22\linewidth]{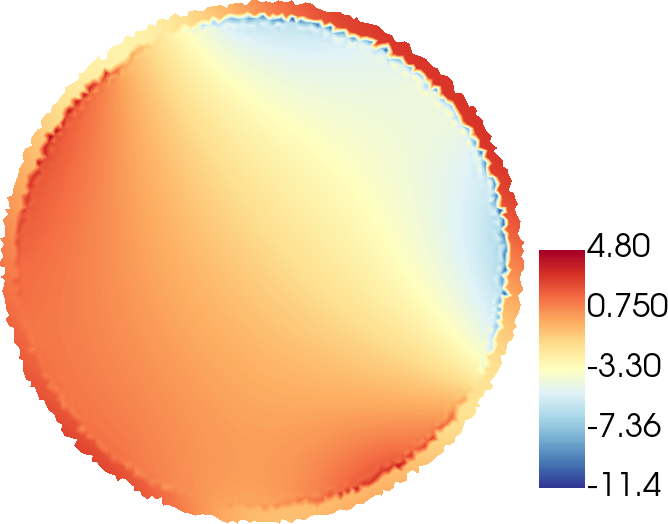} &
        \includegraphics[width=0.22\linewidth]{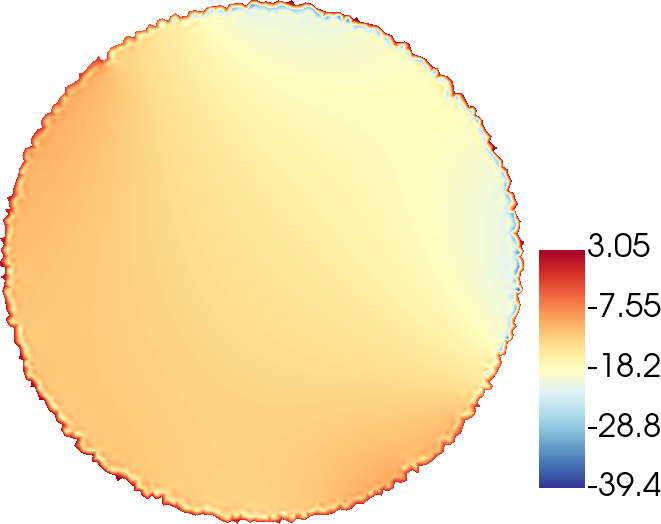}
    \end{tabular}  
    \caption{Pressure solutions of the coupled Stokes problem drawn over $\Omega_{\text{sub}}$ 
    under varying volume $|\Delta|$ of the boundary region $\Delta \subset \Omega_{\text{sub}}$ 
    where the subproblem pressure Dirichlet boundary condition. The subproblem boundary condition perturbation is large ($\delta = 2$).\label{fig:p-delta}}
\end{figure}

\subsection{Velocity divergence norms under $|\Delta|$-refinement}

We are testing how well the discrete incompressibility condition $\nabla \cdot \bm u_h = 0$ is satisfied when solving the Stokes problem with Dirichlet boundary conditions for the velocity, combined with one of the following pressure boundary conditions: (i) the pressure boundary condition from \eqref{eq:bc_u_p_homogenous_reformulated}, or (ii) the zero average pressure condition $\int_\Omega p_h\, d\Omega = 0$. 

Our goal is to understand how the size of  $\Delta$ affects two measures of incompressibility: the velocity divergence $\|\nabla \cdot \bm u_h\|_{L^1(\Omega)}$, and the absolute average of the velocity divergence, $|\int_\Omega \nabla \cdot \bm u_h\, d\Omega|$. Both measures are preferred to be as close to zero as possible. Additionally, we aim to investigate how these two measures differ across the two cases (i) and (ii).
 
\begin{figure}[t]
    \centering
\begin{tabular}{ccc}
    \hspace{1cm}$L^1$-norm & \hspace{0.7cm}Abs. average & \\
    \includegraphics[width=0.3\linewidth]{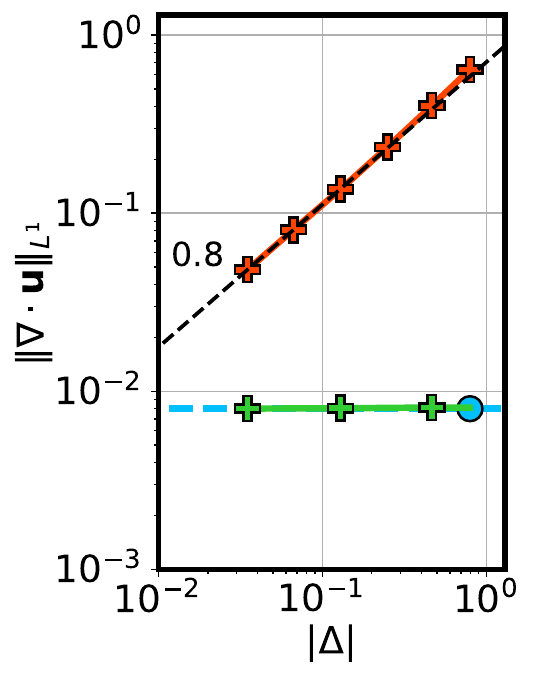} &
    \includegraphics[width=0.287\linewidth]{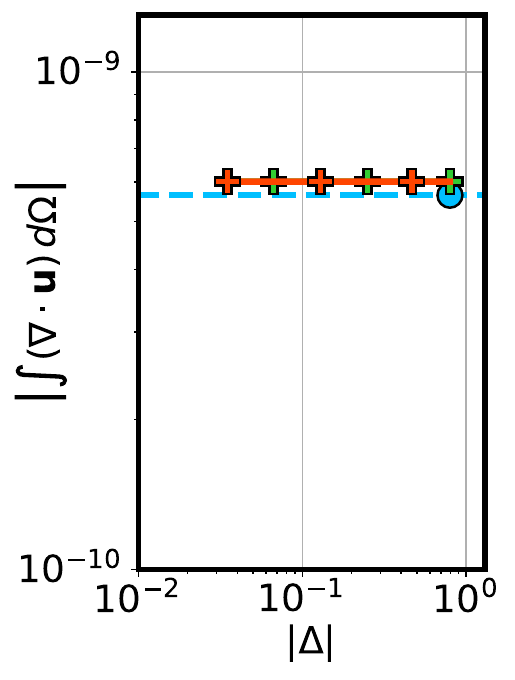} &
    \raisebox{2cm}{\includegraphics[width=0.22\linewidth]{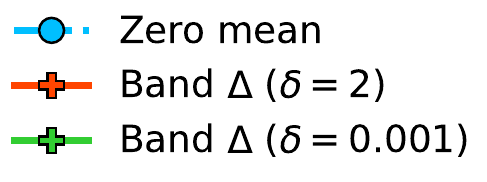}}
\end{tabular}
\caption{Velocity divergence norms under $|\Delta|$-refinement, when Dirichlet boundary conditions are used for velocity and further supplemented by 
the following constraints in pressure: (i) $\int_\Omega p\, d\Omega = 0$ (label Zero mean), 
(ii) $p|_\Delta = p_*|_\Delta + p_\delta$ (label Band $\Delta$). Here $p_*|_\Delta$ is the pressure solution to (i) evaluated at $\Delta \subset \Omega$. 
The function $p_\delta$ is a perturbation with magnitudes $\delta=0.001$ (small perturbation), $\delta=2$ (large perturbation).}
\label{fig:experiments:div_velocity}
\end{figure}

The computed solution pair $(\bm u_h, p_h)$ generally differs between conditions (i) and (ii) since these conditions define two distinct Stokes models. Consequently, the magnitudes of the discrete incompressibility condition $\nabla \cdot \bm u_h$ 
also vary across the two cases. To enable a meaningful comparison, we establish a connection between conditions (i) and (ii) as follows.
First, we solve the Stokes problem using the constraint $\int_\Omega p_h^*\, d\Omega = 0$ (condition (ii)), obtaining the solution pair $(\bm u_h^*, p_h^*)$. 
We then reuse $p_h^*$ to define a new problem under condition (i), modifying the boundary conditions to $p_h|_\Delta = p_h^*|_\Delta + p_h^\delta|_\Delta$, where $p_h^\delta|_\Delta$ is a pressure perturbation function given by $p_h^\delta = \delta \cos (2\pi x_1\, x_2)$. We consider both a small perturbation ($\delta=10^{-3}$) and a large perturbation ($\delta=2$), with $\Delta \subset \Omega$ defined as a band surrounding $\partial\Omega$. 
For our setup, we take $\Omega=(0,1)^2$, $\bm f = (0, -1)$, and impose the boundary condition
$\bm u_h|_{\partial\Omega} = (x \sin (2\pi\, x_1\, x_2), \sin (2\pi\, x_1\, x_2))$, $(x_1, x_2) \in \partial\Omega$. 
The results are presented in Figure \ref{fig:experiments:div_velocity}. 

When the perturbation in the boundary condition is large, we observe that as $|\Delta|\to 0$, the norm $\|\nabla \cdot \bm u_h\|_{L^1(\Omega)}$ converges approximately as $|\Delta|^{0.8}$ towards the zero average pressure case. This rate of convergence exceeds the expectations set by our a priori estimate in Lemma \ref{lemma:divfree:divu_L2_DeltaOnly}. We attribute this behaviour to the regularity of the chosen forcing term $\bm f=(0,-1)$, which is infinitely smooth, whereas our estimate is designed to accommodate forcing terms in the space $\bm H^{-1}(\Omega)$.

When the perturbation in the pressure boundary data is small, we observe that $\|\nabla \cdot \bm u_h\|_{L^1(\Omega)}$ 
remains small and identical to the zero-average pressure case for all values of $|\Delta|$. Notably,
$\|\nabla \cdot \bm u_h\|_{L^1(\Omega)}$ does not decay at the expected order of 0.5 when $|\Delta| \to 0$, as stated in Lemma \ref{lemma:divfree:divu_L2_DeltaOnly}.
This discrepancy arises due to numerical approximation errors in $\nabla \cdot \bm u_h$ when computed under the zero average pressure constraint. 

The term $|\int_\Omega \nabla \cdot \bm u_h\, d\Omega|$ remains at the level of round-off errors for all considered pressure constraint types, regardless of the perturbation size or the choice of $|\Delta|$. This result aligns with the discussion in Section \ref{section:Stokes_wellposedness_pressure_L_2_00}.

\section{Guidelines on imposing Dirichlet-type pressure boundary conditions}\label{sec:guidelines}

We now use the Stokes well-posedness estimate from Theorem \ref{theorem:stokes_wellposedness_final} along with the numerical results from Section \ref{section:numerical_experiments} to provide guidelines for imposing Dirichlet-type pressure boundary conditions in different scenarios. When combined with Dirichlet velocity boundary conditions, we find that:
\begin{itemize}
\item $p_h|_\Delta = 0$ \textbf{is recommended} when $\Delta$ does not scale with mesh size,
\item $p_h|_{\partial\Omega}=0$ \textbf{is not recommended},
\item $p_h(\bm x)=0$ for a single point $\bm x\in\partial\Omega$ or $\bm x\in\Omega$ \textbf{is not recommended}.
\end{itemize}

The first case represents the standard setting when (i) $\Delta \subset \Omega$ and $|\Delta| > 0$. According to Theorem \ref{theorem:stokes_wellposedness_final} we have $C_\Delta^{-1} < \infty$ and well-posedness.

The second and third cases arise in special situations: 
(ii) when $\Delta$ is a ring around $\partial\Omega$ and $|\Delta| \to 0$, or (iii) when $\Delta$ is a ball centered at $\bm x\in\overline{\Omega}$ and $|\Delta| \to 0$. 
In these limiting cases, Theorem \ref{theorem:stokes_wellposedness_final} implies that $C_\Delta^{-1} \to \infty$, indicating a loss of well-posedness. 
This aligns with our numerical experiments, where we observed that $\|p_h\|_{L^2(\Omega)} \to \infty$ as $h \to 0$ when the boundary condition data included a large perturbation. However, for small perturbations, we did not observe unbounded growth of $\|p_h\|_{L^2(\Omega)}$ as $h\to 0$. 
Thus, while there are cases where the pressure boundary condition $p_h|_{\partial\Omega}=0$ does not lead to spurious pressure growth, the condition $p_h|_\Delta = 0$ has the distinct advantage of preventing such growth entirely, regardless of the perturbation size. In PDE model coupling methods \cite{ISSM_iscal,iscal} and domain decomposition methods \cite{domain_decomposition_toselli2006}, boundary condition data inherently contain model errors. For robustness, we therefore recommend using $p_h|_\Delta = 0$.

Another important consideration is the choice of $\Delta$. When $\Delta$ is mesh dependent, the size of $\Delta$ scales as $|\Delta|\propto h^d$ in $d$ dimensions. 
This leads to $C_\Delta \propto |\Delta| \propto h^d$, which causes 
$C_\Delta^{-1} \to \infty$ as $h \to 0$. To avoid this issue, we recommend always choosing $\Delta$ independent of the mesh size, with its size proportional to the size of $\Omega$. This ensures that  $C_\Delta^{-1}$ remains finite in Theorem \ref{theorem:stokes_wellposedness_final} thereby maintaining well-posedness. 

\section{Final remarks}\label{sec:finalremarks}

In this paper, we investigated the well-posedness of the Stokes problem \eqref{eq:pstokes} under Dirichlet boundary conditions, as stated in \eqref{eq:bc_u_p_homogenous_reformulated}. Specifically, we considered the case where the velocity vanishes on $\partial\Omega$ and the pressure vanishes over a region $\Delta \subset \Omega$ with $|\Delta|>0$. The associated infinite-dimensional function spaces for velocity and pressure were $\bm V = \bm H^1_0(\Omega)$ and $Q = L^2_\Delta(\Omega)$, respectively.

We found that the Stokes problem is well-posed, with the a priori pressure estimate depending on $\Delta$, as stated in Theorem \ref{theorem:stokes_wellposedness_final}. This result arises from our extension of the Nečas inequality to $L^2_\Delta(\Omega)$ spaces, as detailed in Theorem \ref{theorem:finalresult_necas_extended}. Specifically, we found that the constant $C_\Delta>0$ in Theorem \ref{theorem:finalresult_necas_extended} scales with $|\Delta|^{1/2}$
when the pressure function approaches a constant. This represents the limiting case where pressure uniqueness is violated. Our numerical experiments confirmed these theoretical observations.
Building on these results, we provided guidelines for imposing pressure Dirichlet boundary conditions when solving the Stokes problem (Section \ref{sec:guidelines}). These guidelines can be directly applied in PDE model coupling methods, such as those in \cite{ISSM_iscal,iscal}.

Future work includes applying the derived well-posedness estimates to the convergence analysis of overlapping domain decomposition methods for Stokes and Navier-Stokes problems. Additionally, the extended Nečas inequality could be used for well-posedness analysis of numerical discretization methods employing discontinuous trial spaces.

\section*{Acknowledgments}

We thank Murtazo Nazarov from Uppsala University and Salvador Rodriguez-Lopéz from Stockholm University 
for valuable discussions. 
Josefin Ahlkrona and Igor To\-mi\-nec were funded by the Swedish Research Council, grant number 2021-04001. 

\section*{Declarations}

Conflict of interest: not applicable.

\bibliographystyle{spmpsci}
\bibliography{refs}
\end{document}